\documentclass[a4paper]{amsart}
\usepackage[latin9]{inputenc}
\usepackage{a4}
\usepackage{amsfonts}
\usepackage{amssymb}
\usepackage{amsmath}
\usepackage{amsthm}
\usepackage{changepage}
\usepackage{nomencl}
\usepackage{geometry}
\usepackage{tikz}
\usepackage{verbatim}
\usetikzlibrary{matrix}
\usepackage{hyperref}
\usepackage[toc,page]{appendix}
\usepackage{mathrsfs}
\usepackage[square, numbers, comma]{natbib}

\usepackage{enumerate}
\usepackage{amsbsy}

\usepackage[all]{xy}
\usepackage{pb-diagram,pb-xy}
   \bibpunct{[}{]}{,}{n}{}{,}
\input cyracc.def

\begin{document}
\providecommand{\keywords}[1]{\textbf{\textit{Keywords: }} #1}
\newtheorem{theorem}{Theorem}[section]
\newtheorem{lemma}[theorem]{Lemma}
\newtheorem{proposition}[theorem]{Proposition}
\newtheorem{corollary}[theorem]{Corollary}
\theoremstyle{definition}
\newtheorem{definition}{Definition}[section]
\theoremstyle{remark}
\newtheorem{remark}{\bf Remark}
\newtheorem{conjecture}{\bf Conjecture}
\newtheorem{question}{\bf Question}
\newtheorem{example}{\bf Example}

\def\p{\mathfrak{p}}
\def\q{\mathfrak{q}}
\def\s{\mathfrak{S}}
\def\Mon{\mathrm{Mon}}
\def\Gal{\mathrm{Gal}}
\def\Ker{\mathrm{Ker}}
\def\Coker{\mathrm{Coker}}
\newcommand{\cc}{{\mathbb{C}}}   
\newcommand{\ff}{{\mathbb{F}}}  
\newcommand{\nn}{{\mathbb{N}}}   
\newcommand{\qq}{{\mathbb{Q}}}  
\newcommand{\rr}{{\mathbb{R}}}   
\newcommand{\zz}{{\mathbb{Z}}}

\title{On sets of rational functions which locally represent all of $\mathbb{Q}$}
\author{Benjamin Klahn}
\author{Joachim K\"onig}
\address{Graz University of Technology, Institute of Analysis and Number Theory, Kopernikusgasse 24/II, 8010 Graz, Austria}
\address{Korea National University of Education, Department of Mathematics Education, 28173 Cheongju, South Korea}
\begin{abstract}
We investigate finite sets of rational functions $\{ f_{1},f_{2}, \dots, f_{r} \}$ defined over some number field $K$ satisfying that any $t_{0} \in K$ is a $K_{p}$-value of one of the functions $f_{i}$ for almost all primes $p$ of $K$. We give strong necessary conditions on the shape of functions appearing in a minimal set with this property, as well as numerous concrete examples showing that these necessary conditions are in a way also close to sufficient. We connect the problem to well-studied concepts such as intersective polynomials and arithmetically exceptional functions.
\end{abstract}

\maketitle

\section{Introduction and main results}
\label{sec:intro}
Given a rational number $\alpha$ and a sufficiently large prime $p$, one of the numbers $\alpha$, $\alpha-1$ and $\frac{\alpha-1}{\alpha}$ is a square in $\mathbb{Q}_p$. Indeed, this is a direct consequence of the multiplicativity of the Legendre symbol. Moreover, this fact translates directly into a fact about rational functions: letting $f_1(X) = X^2$, $f_2(X) = X^2+1$ and $f_3(X) = \frac{1}{1-X^2}\in \mathbb{Q}(X)$, for every $\alpha\in \mathbb{Q}$ and every sufficiently large prime $p$ (depending on $\alpha$), there exists $\beta\in \mathbb{Q}_p$ such that one of $f_1(\beta)$, $f_2(\beta)$ and $f_3(\beta)$ equals $\alpha$. In this paper, we will investigate this curious property of a given finite set of rational functions in greater generality. To this end, we first make some definitions.

For a field $F$ and a rational function $f\in F(X)$, say that $\alpha\in F$ is an $F$-value of $f$ if there exists $\beta\in F\cup \{\infty\}$ such that $f(\beta)=\alpha$.

\begin{definition}
Let $K$ be a number field and $f_1,\dots, f_r\in K(X)$ rational functions. 
We say that $t_0\in K$ is a pseudo-value of $\{f_1,\dots, f_r\}$ if 
there exists a finite set $S$ of primes of $K$ such that for all primes $p$ of $K$ outside of $S$, there exists some $i\in \{1,\dots, r\}$ such that $t_0$ is a $K_p$-value of $f_i$. If $t_0$ is a pseudo-value of $\{f_1,\dots, f_r\}$, but not a $K$-value of any $f_i$ ($i=1,\dots, r$), we say that $t_0$ is a fake value of $\{f_1,\dots, f_r\}$. 
Say that $f_1,\dots, f_r$ locally represent $K$, if all elements of $K$ are pseudo-values of $f_1,\dots, f_r$. Finally, call such $\{f_1,\dots, f_r\}$ a minimal locally representing set if no proper subset locally represents $K$.
\end{definition}

Of course, a single rational function of degree $1$ locally represents its field of definition, since it even (globally) attains any value. In contrast to this, the sets of $K$-values of finitely many rational functions of degree $>1$ over a number field $K$ can never cover all of $K$; this allows locally representing sets to be interpreted as counterexamples to a local-global principle. The idea of locally representing sets can be traced back to the work of Fried, notably \cite[(4.5) and Example 7]{Fried74}, which, within a more general investigation of mod-$p$ value sets, briefly mentions a very similar (and after all equivalent) property and gives a class of examples.   
Note also that for $r=1$, the definition of fake values coincides with the one in \cite{Corvaja}. That paper also identifies situations in which the set of fake values of a rational function (or a more general cover) is infinite. 
The property of sets of (non-linear) rational functions to locally represent all of $K$ may be seen as the strongest possible form of this phenomenon; note that this can no longer happen with $r=1$, see Corollary \ref{cor:transl3}b). We propose the following general question.

\begin{question}
\label{ques:2}
Let $K$ be a number field. 
 What are the minimal locally representing sets $\{f_1,\dots, f_r\}$ of geometrically indecomposable rational functions $f_i\in K(X)$?
\end{question}

Here, a rational function $f\in K(X)$ is called {\it indecomposable} (resp., {\it geometrically indecomposable}) if it cannot be written as a composition $f = g\circ h$ of two functions $g,h \in K(X)$ (resp. $\in \overline{K}(X)$) of degree $>1$. 
Of course Question
~\ref{ques:2} may also be asked without the indecomposability assumption. It is nevertheless a reasonable assumption in this context, as explained in Section \ref{sec:firstobs} (see in  particular Lemma \ref{lem:firstobs}  and Corollary \ref{cor:transl2}),

While we do not attempt to answer Question \ref{ques:2} in full, we will reasonably classify the possible shapes of the individual functions $f_i$. 
After having seen an easy example of how several rational functions combined may ``work together" to locally represent all of $\mathbb{Q}$, it may come as a surprise that ``most" rational functions are in fact useless for creating such examples. This is the content of our first main result Theorem \ref{thm:main}.

To state our results we recall some standard notions for rational functions. For a rational function $f \in K(X)$ defined over some number field $K$ write $f=g/h$ where $g,h \in K[X]$ are coprime. Then we denote by $\mathrm{Split}(f(X)-t/K(t)):=\mathrm{Split}(g(X)-th(X)/K(t))$ the \textit{splitting field} of $f(X)-t$. 
\begin{theorem}
\label{thm:main}
Given any $r\in \mathbb{N}$, there exists a constant $N:=N(r)\in \mathbb{N}$ such that the following holds: Let $K$ be a number field, and let $f_1,\dots f_r\in K(X)$ be a minimal collection of geometrically indecomposable rational functions locally representing $K$.
Then each $f_i$ fulfills one of the following:
\begin{itemize}
\item[1)] $\deg(f_i) \le N$, or 
\item[2)]$\mathrm{Split}(f_i(X)-t/K(t))$ is of genus $\tilde{g}\le 1$.
\end{itemize}
\end{theorem}

The proof of Theorem \ref{thm:main} rests on a combination of group-theoretical arguments based on the concept of ``intersective polynomials" with classification results for monodromy groups of indecomposable rational functions, initiated by the ``Guralnick-Thompson conjecture" (see \cite{GT}) on composition factors of genus zero monodromy groups. 

Recent (and so far unpublished) results by Neftin and Zieve (see \cite{NZ}) allow to conclude that the constant $N$ in Theorem \ref{thm:main} is in fact absolute, i.e., that the following holds:

{\bf Theorem} \ref{thm:main}'.
There exists an absolute bound $N\in \mathbb{N}$ such that the following holds: Let $K$ be a number field, let $r\in \mathbb{N}$ and let $f_1,\dots f_r\in K(X)$ be a minimal collection of geometrically indecomposable rational functions locally representing $K$.
Then each $f_i$ fulfills one of the following:
\begin{itemize}
\item[1)] $\deg(f_i) \le N$, or 
\item[2)]$\mathrm{Split}(f_i(X)-t/K(t))$ is of genus $\tilde{g}\le 1$.
\end{itemize}

Note that the condition in Case 2) for the splitting field to be of genus $\le 1$ is very restrictive. The respective rational functions have been explicitly classified, and in particular it is known that the monodromy group $\mathrm{Gal}(f_i(X)-t/K(t))$ is either solvable or embeds into $S_5$, cf., Sections \ref{sec:gal_genus0} and \ref{sec:gal_genus1}. 
In particular, Theorem \ref{thm:main}' implies that ``generic" rational functions, namely such with monodromy group $S_n$, cannot be part of a set as in Question \ref{ques:2}, except in certain small degrees. For this ``generic" case, we can however make the result much more explicit.

\begin{theorem}
\label{thm:sym}
Let $K$ be a number field, and let $f_1,\dots, f_r\in K(X)$ be rational functions of degree $d_1,\dots, d_r$ with $\mathrm{Gal}(f_i(X)-t/K(t)) = S_{d_i}$ or $A_{d_i}$ ($i=1,\dots, r$). If $\{f_1,\dots, f_r\}$ is a minimal locally representing set, then all $d_i$ are in $\{2,3,4,6\}$.
\end{theorem}

In Section \ref{sec:exist}, we will complement the above theorems by existence results showing that the types of functions not yet excluded in Theorems \ref{thm:main} and \ref{thm:sym} do indeed occur.

Namely, we show the following partial converse to Theorem \ref{thm:main}.

\begin{theorem}
\label{thm:converse}
Let $f\in \mathbb{Q}(X)$ be a geometrically indecomposable rational function of degree $>163$ such that the splitting field of $f(X)-t$ is of genus $\le 1$. Then 
there exists $\tilde{f}\in \mathbb{Q}(X)$, linearly related to $f$ over $\mathbb{C}$, such that $\tilde{f}$
is part of a minimal set of rational functions locally representing $\mathbb{Q}$.
%
%
%
%
\end{theorem}

Here, two rational functions $f_1, f_2\in K(X)$ are called {\it linearly related} to each other (over $K$) if there exist fractional linear transformations $\lambda, \mu\in K(X)$ (i.e., rational functions of degree $1$) such that $f_2 = \lambda\circ f_1\circ \mu$.

\section{Preliminaries}

 \subsection{Monodromy, inertia and decomposition groups of rational functions} 
\label{sec:ratfcts}
We collect some facts on rational functions and their monodromy, which will be useful both for the proofs in Section \ref{sec:mainproofs} and for the existence results in Section \ref{sec:exist}.

Let $K$ be a field of characteristic $0$ and $f(X)=\frac{g(X)}{h(X)}\in K(X)$ a rational function where $g$ and $h$ are coprime polynomials over $K$. The {\it degree} of $f$ is the maximum of $\deg(g)$ and $\deg(h)$. The roots of $g(X)-th(X)$ are referred to as the roots of $f(X)-t$. The (arithmetic) {\it monodromy group} of $f$ is the group $\Mon(f) := \Mon_K(f):= \mathrm{Gal}(g(X)-th(X)/K(t))$. The {\it geometric monodromy group} $\Mon_{\overline{K}}(f)$ of $f$ is the Galois group of the same polynomial over $\overline{K}(t)$; it is a normal subgroup of $\Mon_K(f)$, since $\overline{K}\cap \mathrm{Split}(f(X)-t/K(t))$ is a normal extension of $K$. Moreover, $\Mon_{\overline{K}}(f)$ is anti-isomorphic to the group of deck transformations of the Galois closure of the cover $\mathbb{P}^1_{\overline{K}}\to \mathbb{P}^1_{\overline{K}}$ given by $x\mapsto f(x)$. This cover has finitely many {\it branch points} $t_1,\dots, t_r\in \mathbb{P}^1_{\overline{K}}$, which are exactly the critical values of the function $f$.
Associated to each branch point $t_i$ is a conjugacy class $C_i$ of the geometric monodromy group, and the cycle lengths of $\sigma_i\in C_i$ (acting on the roots of $f(X)-t$) are the ramification indices of places extending $t\mapsto t_i$ in a root field of $f-t$, which in term are simply the multiplicities of points in $f^{-1}(t_0)\subset \mathbb{P}^1_{\overline{K}}$. See \cite[Lemma 3.1]{Mue_hit}. The tuple $(C_1,\dots, C_r)$ is called the {\it ramification type} of the rational function $f$; moreover, one can choose a tuple $(\sigma_1,\dots, \sigma_r)$ with $\sigma_i\in C_i$ ($i=1,\dots, r$), such that $\sigma_1\cdots\sigma_r=1$ and $\langle\sigma_1,\dots, \sigma_r\rangle= \Mon_{\overline{K}}(f)$. The tuple $(\sigma_1,\dots, \sigma_r)$ is called a {\it branch cycle description} of $f$; the cyclic group $\langle\sigma_i\rangle$ is an {\it inertia group} at the branch point $t\mapsto t_i$ in the splitting field of $f(X)-t$.
 The tuple $(\sigma_1,\dots, \sigma_r)$ is moreover a ``genus-zero tuple", which via the Riemann-Hurwitz formula can be expressed in the form of an explicit condition on the total number of cycles of $\sigma_1,\dots, \sigma_r$. To state this precisely, recall the notion of the {\it index} of a permutation $\sigma\in S_n$; it is defined as $\mathrm{ind}(\sigma):= n - \#\{\text{ orbits of } \langle\sigma\rangle\}$.

\begin{proposition}[Riemann-Hurwitz formula (for the case of rational functions)]
\label{prop:rh}
Let $f\in K(X)$ be a rational function of degree $n$ with ramification type $(C_1,\dots, C_r)$. Then for $\sigma_i\in C_i$, $i=1,\dots, r$, one has
$$2n-2 = \sum_{i=1}^r \mathrm{ind}(\sigma_i).$$
\end{proposition}

Conversely, one has the following, as a special case of Riemann's existence theorem.
\begin{proposition}
\label{prop:ret}
Let $\sigma_1,\dots, \sigma_r\in S_n$ such that $\sigma_1\cdots\sigma_r=1$, the subgroup $G:=\langle \sigma_1,\dots, \sigma_r\rangle\le S_n$ is transitive and $2n-2 = \sum_{i=1}^r \mathrm{ind}(\sigma_i)$. Then there exists a rational function $f$, defined over some number field $K$, such that the geometric monodromy group of $f$ equals $G$ and a branch cycle description is given by $(\sigma_1,\dots, \sigma_r)$.
\end{proposition}

Determining the difference between arithmetic and geometric monodromy group furthermore amounts to quantifying the extension of constants in the splitting field of $f(X)-t$. For this, the pair $(I,D)$ of inertia and {\it decomposition group} at a branch point often provides useful information. The following is, e.g., a special case of \cite[Prop.\ 7.1.2]{Serre}.
\begin{proposition}
%
%
\label{prop:bcl}
Let $f\in K(X)$ be a rational function, and let $\lambda\in K\cup\{\infty\}$ be a $K$-rational branch point of $x\mapsto f(x)$, with inertia group $I$ of order $e$. Let $A$ be the (arithmetic) monodromy group of $f$. Then there exists a subgroup $D\le A$ with $I\trianglelefteq D$ such that $D/I$ surjects onto $\mathrm{Gal}(K(\zeta_e)/K)$.
\end{proposition}

Finally, recall that many notions related to composition of rational functions (cf.\ Section \ref{sec:intro}) translate naturally into permutation-group theoretical notions on the side of the monodromy group. In particular, a rational function $f\in K(X)$ is indecomposable (resp.,  geometrically indecomposable), if and only if its monodromy group (resp., its geometric monodromy group) is a primitive permutation group. Similarly, two functions $f_1,f_2\in K(X)$ which are linearly related over $K$ (i.e., $f_2=\lambda\circ f_1\circ \mu$ for functions $\lambda,\mu\in K(X)$ of degree $1$) have the same monodromy group, as well as the same ramification type.

\subsection{Classification results for monodromy groups of rational functions}

Being the monodromy group of a rational function is in fact quite a restrictive condition on a group $G$. We review various results related to the classification of such monodromy groups.

\subsubsection{Rational functions with Galois closure of genus $0$}
\label{sec:gal_genus0}
The classification of genus-$0$ {\it Galois} extensions of function fields in characteristic $0$ is classical. It follows essentially from the Riemann-Hurwitz formula and was known to Klein. See, e.g., \cite[Theorem I.6.2]{MM}.
\begin{proposition}
\label{prop:gal_genus0}
Let $K$ be a field of characteristic $0$, and $f\in K(X)$ be a rational function of degree $>1$ such that the splitting field of $f(X)-t$ over $K(t)$ is of genus $0$. Let $G$ be the geometric monodromy group of $f$, and $e_1,\dots, e_r$ be the ramification indices at branch points in $\mathrm{Split}(f(X)-t)/K(t)$. Then one of the following holds:
\begin{itemize}
\item[a)] $G\cong C_n$, $r=2$ and $(e_1,e_2)=(n,n)$ for some $n\ge 2$. In this case, $f$ is linearly related over $\mathbb{C}$ to $X^n$.
\item[b)] $G\cong D_n$, $r=3$ and $(e_1,e_2,e_3)=(n,2,2)$ for some $n\ge 2$ (with $D_2:=V_4$). In this case, $f$ is linearly related over $\mathbb{C}$ to the degree-$n$ Chebyshev polynomial $T_n$ or the degree-$2n$ rational function $X^n+\frac{1}{X^n}$.
\item[c)] $G\cong A_4$, $r=3$ and $(e_1,e_2,e_3)=(2,3,3)$.
\item[d)] $G\cong S_4$, $r=3$ and $(e_1,e_2,e_3)=(2,3,4)$.
\item[e)] $G\cong A_5$, $r=3$ and $(e_1,e_2,e_3)=(2,3,5)$.
\end{itemize}
\end{proposition}

\begin{remark}
\label{rem:gal_genus0}
\begin{itemize}
\item[a)] Furthermore, the arithmetic monodromy group $\Mon_K(f)$ of a function $f$ as in Proposition \ref{prop:gal_genus0} is very restricted, since necessarily contained in the symmetric normalizer of the geometric monodromy group. E.g., in Cases a) and b), this symmetric normalizer equals the group $AGL_1(\mathbb{Z}/n\mathbb{Z}) = C_n\rtimes \mathrm{Aut}(C_n)$ of order $n\cdot \varphi(n)$.
\item[b)] The functions in Cases c)-e) can also be described more explicitly. They all correspond to subcovers of the tetrahedral, octahedral or icosahedral Galois covers $\mathbb{P}^1\to \mathbb{P}^1$, and have all been computed explicitly, see, e.g., \cite[Theorem 1.1]{Pak}.
\item[c)] The list of functions shrinks even further if one adds the requirement of indecomposability, or equivalently, of the monodromy group acting primitively. E.g., in Cases a) and b), the only remaining functions are those linearly related over $\mathbb{C}$ to $X^p$ (for a prime $p$) or to $T_p$ (for an odd prime $p$), since the functions $X^n+\frac{1}{X^n}$ correspond to the regular, hence imprimitive action of $D_n$.
\end{itemize}
\end{remark}

\subsubsection{Rational functions with Galois closure of genus $1$}
\label{sec:gal_genus1}

In a similar way, rational functions whose Galois closure is of genus $1$ have been known quite explicitly (i.e., via their group theoretical data) for a long time. They are associated to endomorphisms or isogenies of elliptic curves. We collect some key facts on these functions from \cite[Section 6]{GMS}.
\begin{proposition}
\label{prop:gal_genus1}
Let $K$ be a  field of  characteristic $0$, and $f\in K(X)$ be a rational function such that the splitting field of $f(X)-t$ over $K(t)$ is of genus $1$. Let $A$ and $G$ denote the arithmetic and geometric monodromy group of $f$. Then all of the following hold:
\begin{itemize}
\item[a)] The tuple of ramification indices at branch points in $\mathrm{Split}(f(X)-t)/K(t)$ is one of $(2,2,2,2)$, $(3,3,3)$, $(2,4,4)$ and $(2,3,6)$.
\item[b)] 
If furthermore $f$ is assumed indecomposable, then $\deg(f)=p$ or $\deg(f)=p^2$ for a prime $p$. In the first case, $G\cong C_p\rtimes C_d$ for $d=2,3,4$ or $6$ (according to the four possibilities in a)) and $G\trianglelefteq A\le AGL_1(p)$. In the second case, $G\cong (C_p\times C_p)\rtimes C_d$ with $d\in \{2,3,4,6\}$, and $G\trianglelefteq A\le AGL_2(p)$.
\end{itemize}
\end{proposition}

One obtains even stronger conclusions upon adding certain extra assumptions on the function $f$ and its field of definition.
\begin{remark}
\label{rem:gal_genus1}
\begin{itemize}
\item[a)] 
If $K=\mathbb{Q}$, then the case $\deg(f)=p$ in Conclusion b) of Proposition \ref{prop:gal_genus1} occurs only for finitely many primes $p$. More precisely, the only possibility for the tuple ramification indices is then $(2,2,2,2)$, and $p\le 163$, see \cite[Theorem 6.7]{GMS}. The reason for the latter conclusion is that such functions arise from elliptic curves over $\mathbb{Q}$ with a $p$-isogeny, and Mazur's theorem \cite{Mazur} yields that this can happen only for an explicitly known list of $p\le 163$.
\item[b)] 
If $f$ is even assumed {\it geometrically} indecomposable, then in the case $\deg(f)=p^2$ of Proposition \ref{prop:gal_genus1}, the tuple $(2,2,2,2)$ of ramification indices is not possible. Indeed, in this case one gets geometric monodromy group $G\cong (C_p\times C_p)\rtimes C_2\le AGL_2(p)$ for an odd prime $p$, and this group cannot act primitively on $p^2$ points. Indeed, an element of order $2$ in $GL_2(p)$ has to have a nontrivial eigenspace in $(\mathbb{F}_p)^2$, yielding an overgroup $C_p\rtimes C_2$ of the point stabilizer $C_2\le G$. For the same reason, it follows in the remaining three cases of ramification tuples $(3,3,3)$, $(2,4,4)$ and $(2,3,6)$ that $p$ cannot be congruent to $1$ modulo $d$. In particular, the geometric monodromy group is of the form $G=(C_p\times C_p)\rtimes C_d$ with a {\it non-central} subgroup $C_d\le GL_2(p)$. Since the normalizer of such a subgroup in $GL_2(p)$ is necessarily solvable, 
it follows that the last assertion in Proposition \ref{prop:gal_genus1} can then be strengthened to $A$ being a {\it solvable} subgroup of $AGL_2(p)$. If additionally $K=\mathbb{Q}$, the shape of the arithmetic monodromy group is determined completely as well. Indeed, $f$ then arises from the multiplication-by-$p$ map on an elliptic curve over $\mathbb{Q}$ with complex multiplication by $\mathbb{Q}(\zeta_d)$, cf.\ \cite[Theorem 6.6]{GMS}. Theorems 6.21, 6.23 (together with Remark 6.27) and 6.25 of \cite{GMS} then yield that for each of the three possibilities of $d$, the group $\Mon(f)$ is of the form $(C_p\times C_p)\rtimes (C_{p^2-1}\rtimes C_2) = A\Gamma L_1(\mathbb{F}_{p^2})$ (with the outer $C_2$ acting on $C_{p^2-1}$ as the field automorphism of $\mathbb{F}_{p^2}$).
%
\end{itemize}
\end{remark}

\subsubsection{Monodromy of large degree rational functions}
\label{sec:largemon}

The full classification of monodromy groups of rational functions has been an ongoing project for several decades. In the following theorem, we summarize the contributions from several authors in a way suitable for our applications. For this, we need to recall the notion of {\it product action}. See, e.g., \cite{MueTwo} or \cite[Section 5]{GN}.  
Set $\Delta=\{1,2,\dots,r\}$ for $r\ge 2$, and let $m\ge 2$ be an integer. Let $S_r \wr S_m := \underbrace{(S_r \times \dots\times S_r)}_{m \text{ times}} \rtimes S_m$ (where $S_m$ acts via permuting the $m$ copies of $S_r$) be the wreath product. The cosets of $S_{r-1}\wr S_m$ in $S_r\wr S_m$ yield a natural action of $S_r\wr S_m$ on $\Omega:= \Delta^m$. We say that a permutation group acts via the product action, if it is permutation equivalent to a transitive subgroup of $S_r \wr S_m$ in this action.

\begin{theorem}
\label{thm:OldClassification}
There is an absolute constant $N \in \mathbb{N}$ such that, if $K$ is a subfield of $\mathbb{C}$ and $f(X) \in K(X)$ is a geometrically indecomposable rational function of degree $n\ge N$, 
then one of the following holds for the monodromy group $\Mon(f)$.
\begin{itemize}
\item[(1)] $\Mon(f)$ is isomorphic, as an abstract group, to an alternating or symmetric group.
\item[(2)] 
$\Mon(f)$ acts as a group of product type, and more precisely $L^t < \Mon(f)\le \mathrm{Aut}(L)^t\wr S_t$ for some $t\in \{2,\dots, 8\}$ where $L$ is isomorphic, as an abstract group, to some alternating group.
\item[(3)] $n=p^i$ for a prime $p$ and an integer $i\le 2$. Furthermore, $\Mon(f) \le AGL_i(p)$ is solvable, and $\mathrm{Split}(f(X)-t/K(t))$ is of genus $\tilde{g}\le 1$.
\end{itemize}
\end{theorem}
The statement of Theorem \ref{thm:main} thus boils down to showing that if $n=\deg(f)$ is sufficiently large and $f$ falls into case (1) or (2) in Theorem \ref{thm:OldClassification}, then $f$ cannot appear in a minimal locally representing set $\{f_{1},f_{2},\dots, f_{r} \}$.
\begin{proof}[Proof of Theorem \ref{thm:OldClassification}]
First, note that the geometric indecomposability assumption amounts to saying that the geometric monodromy group $\Mon_{\overline{K}}(f)$ acts as a primitive group.
Reducing the precise possibilities for the group $\Mon_{\overline{K}}(f)$ inside the classification of all primitive permutation groups (known as the ``O'Nan-Scott theorem") amounts to combining the results of several authors. A short and concise overview is given in \cite{NZ}, ``Proof of Theorem 1.1". 
The case ``$\Mon_{\overline{K}}(f)$ solvable" is the easiest, and reduces readily to the additional requirement that the splitting field of $f(X)-t$ is of genus $\le 1$, whence one arrives at the assertion of 3) via the results of Sections \ref{sec:gal_genus0} and \ref{sec:gal_genus1}. For the remaining cases, since we do not wish to go into the details of the O'Nan-Scott theorem, we content ourselves with reiterating that, by the combined results of \cite{Asch}, \cite{GN}, \cite{GT} and \cite{Shih}, the group $G:=\Mon_{\overline{K}}(f)$ is either an almost simple group or acts as a product type group $L^t < G\le \mathrm{Aut}(L)^t\wr S_t$ for some nonabelian simple group $L$ and some $t\in \{2,\dots, 8\}$. Furthermore, by a result of Frohardt and Magaard (\cite{FM}, proving a famous conjecture of Guralnick and Thompson), the composition factors of $G$ are necessarily all cyclic or alternating  for $n$ sufficiently large. This yields Case 1) or 2) with $\Mon(f)$ so far replaced by its normal subgroup $\Mon_{\overline{K}}(f)$. Note however, that the symmetric centralizer of a nonabelian primitive permutation group is necessarily trivial, see \cite[Theorem 4.2.A(vi)]{DM}. This implies that $\Mon(f) \trianglerighteq \Mon_{\overline{K}}(f)$ embeds into the automorphism group of $\Mon_{\overline{K}}(f)$, and is thus still of the same type (i.e., product type with the same parameter $t$, or almost simple) as the former. Hence, the assertion follows for $\Mon(f)$.
\end{proof}

\begin{remark}
\label{rem:largemon}
As an immediate consequence, there exists a unique minimal finite set $\mathcal{S}$ of non-abelian simple groups such that every geometrically indecomposable rational function $f$ over a number field whose monodromy group has some nonsolvable composition factor outside of $\mathcal{S}$ falls into Case (1) or Case (2) of Theorem \ref{thm:OldClassification}. 
\end{remark}
As mentioned in the introduction, the most recent (and so far only available in preprint form) results on monodromy groups of geometrically indecomposable covers of bounded genus, will allow us to prove a stronger version of Theorem \ref{thm:main}. The below was conjectured by Guralnick-Shareshian in \cite{GS}, and is shown (in even more generality) by Neftin and Zieve in \cite{NZ}. One of its main achievements is the determination not only of possible monodromy groups as abstract groups, but also of all possible ``low genus permutation actions". 
For the case of genus-$0$ actions, corresponding to monodromy groups of rational functions, the result reads as follows:
\footnote{Note that \cite{NZ} gives results in term of the geometric monodromy group. The conclusion for the arithmetic monodromy group, however, follows exactly as in the proof of Theorem \ref{thm:OldClassification}.}
%
%
\begin{theorem}[\cite{NZ}, Theorem 1.1] \label{NZTheo} 
There is an absolute constant $N \in \mathbb{N}$ such that, if $K$ is a subfield of $\mathbb{C}$ and $f(X) \in K(X)$ is a geometrically indecomposable rational function of degree $n\ge N$, 
then one of the following holds.
\begin{itemize}
\item[(1)] $\mathrm{Mon}(f)=A_{n}$ or $S_{n}$ in the natural degree $n$ action,
\item[(2)] $\mathrm{Mon}(f)=A_{\ell}$ or $S_{\ell}$ where $n = \ell(\ell-1)/2$ with the action on the roots of $f(X)-t$ induced by the action on $2$-element subsets of $\{1,2,\dots, \ell \}$,
\item[(3)] $A_{\ell}^{2} < \mathrm{Mon}(f) \leq S_\ell \wr C_2 = S_{\ell}^2 \rtimes C_{2}$, acting in the product action of degree $n=\ell^{2}$,
\item[(4)] $n=p^i$ for a prime $p$ and integer $i\le 2$, and $\mathrm{Mon}(f) \le AGL_i(p)$ is solvable.
\end{itemize}
\end{theorem}

\subsection{Intersectivity, exceptionality and related group-theoretic notions}
\label{sec:intersect}

As we will see in Section \ref{sec:firstobs}, the notion of locally representing sets is closely connected with the following group-theoretical notion, which has been studied frequently (sometimes with slightly differing terminology), see, e.g., \cite{BP} or \cite{BPS}.
\begin{definition}
\label{defi:normal_cov}
Let $G$ be a finite group. A normal covering of $G$ is a collection of proper subgroups $U_1,\dots, U_r\subsetneq G$ such that $\cup_{x\in G}\cup_{i=1}^r U_i^x = G$, i.e., every element of $G$ lies in a conjugate of some $U_i$. A normal covering consisting of exactly $r$ subgroups is also called a normal $r$-covering.
\end{definition}
Via action on the union of the coset spaces of all $U_i$, $i=1,\dots, r$, a normal covering gives rise to a permutation action of $G$ in which every element of $G$ fixes at least one point, whereas $G$ itself has no fixed point.

One reason for the relevance of normal coverings is the following application in Galois theory:
Let $G=\mathrm{Gal}(f/K)$ be the Galois group of a separable polynomial $f=f_1\cdots f_r$ over a number field $K$ (with irreducible polynomials $f_i$ of degree $\ge 2$ for all $i$), and let $U_i\le G$ be the stabilizer of a root of $f_i$ ($i=1,\dots r$).  Then $U_1,\dots, U_r$ being a normal covering of $G$ translates to the property that every element of $G$ fixes a root of $f$. Due to Chebotarev's density theorem, the latter is equivalent to $f$ having a root in all but finitely many completions $K_p$ of $K$; compare \cite[Proposition 1]{Sonn}. Polynomials with this property are known in the literature as (weakly)\footnote{Usually, ``intersective" is understood to imply existence of roots in {\it all} $K_p$, whereas ``weakly intersective" refers to the same property for almost all $K_p$. Since the main notions in our paper deal with properties for almost all primes, we will use the term ``intersective" instead of ``weakly intersective" for convenience.} intersective polynomials. Furthermore, $f$ is called newly intersective if it has no proper intersective divisors. Such polynomials were studied, e.g., in \cite{BB} or \cite{Sonn}.

Since a finite transitive permutation group necessarily possesses a fixed point free element, an irreducible polynomial cannot be intersective. On the other hand, it is straightforward to see that for every non-cyclic group $G$ which occurs as a Galois group over $K$, there exists an intersective polynomial $f=f_1\cdots f_r\in K[X]$ with Galois group $G$, and the smallest number $r$ of factors of such $f$ is directly related to the ``normal covering number" of $G$, i.e., the smallest number $r$ of proper subgroups of $G$ such that $G$ has a normal $r$-covering, cf.\ \cite{BP}. While being the Galois group of an intersective polynomial is thus not a restrictive condition, recall that being the monodromy group of a rational function is quite restrictive, and it is the combination of these two that will become relevant for us.


Our notions also relate naturally to the concept of ``exceptional polynomials" (or more generally exceptional covers). Recall that a polynomial $f\in K[X]$ over a number field is called (arithmetically) exceptional if $f$ induces a bijection on infinitely many residue fields of primes of $K$. This property, once again, translates nicely to a permutation-group theoretic property. In particular, exceptionality necessarily yields a pair $(G,N)$ of permutation groups with the following property, see, e.g., \cite[Lemma 3.3]{GMS}:

(*) $N\trianglelefteq G (\le S_n)$ are two subgroups with the same orbits, and there exists a coset of $N$ in $G$ in which every element fixes a point.

Note that in the classical setup of arithmetically exceptional functions, $G$ and $N$ in (*) are actually transitive, namely equal to the arithmetic and geometric monodromy group of the given function $f$. The generalization to not necessarily irreducible covers has been considered in depth in \cite{Fried}.

The group-theoretical condition (*) is naturally related to normal coverings and intersectivity; indeed, the special case of (*) where $N=G$ and $G$ itself has no fixed point translates to the point stabilizers providing a normal covering of $G$. In the proof of Theorem \ref{thm:main} in Section \ref{sec:mainproofs}, Condition (*) will however also appear with $N<G$. We thus recall some important results on normal coverings of cosets $\sigma N$ by subgroups of certain groups $G$.

\begin{proposition}
\label{prop:groupcover}
There exists a constant $c>0$ such that the following holds. Let $G=S_n$ or $A_n$, let $1\ne N\trianglelefteq G$, and let $\sigma\in G$ and $U_1,\dots, U_r\le G$ be subgroups not containing $A_n$ such that every element of $\sigma N$ lies in a conjugate of some $U_i$, $i=1,\dots, r$. 
Then $r> cn$.
\end{proposition}
\begin{proof} For $N=G$, this amounts exactly to the result on normal covering numbers of symmetric and alternating groups shown in \cite{BPS}. The only other case $N=A_n$, $G=S_n$ follows instantly upon noting that in the case $\sigma\in A_n$, the $2r$ groups $U_i\cap A_n$, $(U_i\cap A_n)^{(1,2)}$ ($i=1,\dots, r$)  yield a normal covering of $A_n$, and for $\sigma\notin A_n$, the $r+1$ groups $U_1,\dots, U_r, A_n$ yield a normal covering of $S_n$.
\end{proof}

The following lemma relates normal coverings of cosets in almost simple permutation groups to such coverings in groups of ``product type" as introduced in Section \ref{sec:largemon}.
\begin{lemma} \label{lem:ProductFixedPointFree}
Let $G$ be a finite almost simple group with socle $L$, let $U_1,\dots, U_r$ be subgroups of $G$ and assume that every coset of $L$ in $G$ contains an element which is  fixed point free in all the actions on cosets of $U_i$ in $G$ ($i=1,\dots, r$).
Then, for any fixed $k\ge 2$, every coset of $L^k$ in the wreath product $G\wr S_k (= G^k \rtimes S_k)$ contains an element which is fixed point free in all the actions of $G\wr S_k$ on cosets of $(U_i)^k \rtimes S_k$ ($i=1,\dots, r$).
\end{lemma}
\begin{proof}
First, note that cosets of $(U_i)^k\rtimes S_k$ can be identified with $k$-tuples $(\omega_1,\dots, \omega_k)$, where the $\omega_j$ are cosets of $U_i$ in $G$. 
Any element $x$ of $G\wr S_k$ can be written as $x=((\sigma_1,\dots, \sigma_k), \tau)$ where $\sigma_1,\dots, \sigma_k\in G$, $\tau\in S_k$, and $x$ acts on cosets of $(U_i)^k\rtimes S_k$ via $(\omega_1\dots, \omega_k)^x = (\omega_{\tau^{-1}(1)}^{\sigma_{\tau^{-1}(1)}}, \dots, \omega_{\tau^{-1}(k)}^{\sigma_{\tau^{-1}(k)}})$. 
Assume now that such  $(\omega_1,\dots, \omega_k)$ is a fixed point of $x$, and for any $j\in \{1,\dots, k\}$, let $d_j$ be the length of the cycle of $\tau$ containing $j$. Considering the action of $x^{d_j}$, one obtains that $\sigma_j \sigma_{\tau(j)}\cdots \sigma_{\tau^{d_j-1}(j)}$ fixes $\omega_j$.

Now fix the $L^k$-coset of $x$ as above. This amounts to fixing $\tau\in S_k$, as well as prescribing the $L$-coset in $G$ of each $\sigma_i$. To show the assertion, it suffices to show the existence of $\sigma_1,\dots, \sigma_k$ in the prescribed cosets  such that for all $j=1,\dots, r$, the element $\sigma_j \sigma_{\tau(j)}\cdots \sigma_{\tau^{d_j-1}(j)}$ is fixed point free. But this just amounts to producing finitely many (namely, one for each cycle of $\tau$) fixed point free elements in prescribed $L$-cosets in $G$, which is possible by assumption.
\end{proof}

 \section{First observations on locally representing sets}
 \label{sec:firstobs}

We now turn to the investigation of locally representing sets of rational functions. 
We begin with some observations on the behavior of such sets under composition, which follow straight from the definitions.
\begin{lemma}
\label{lem:firstobs}
Assume $f_1,\dots, f_r \in K(X)$ locally represent $K$. Then the following hold:
\begin{itemize}
\item[a)]
If  $\lambda$ as well as $\mu_1,\dots, \mu_r\in K(X)$ are fractional linear transformations, then
$\lambda\circ f_1\circ \mu_1, \lambda\circ f_2\circ \mu_2, \dots, \lambda\circ f_r\circ \mu_r$ locally represent $K$.
\item[b)] If $f_i = g_i\circ h_i$ are decompositions of $f_i$ into rational functions $g_i, h_i\in K(X)$, then $g_1,\dots, g_r$ locally represent $K$.
\end{itemize}
\end{lemma}

The following characterization of locally representing sets of rational functions will be crucial for the further treatment. The equivalences 1) to 3) demonstrate the extent to which locally representing and non-representing sets of functions differ from each other, whereas 4) and 5) translate the arithmetic property into a group-theoretical one, relating directly to the notions of normal coverings and  intersectivity discussed in Section \ref{sec:intersect}.
\begin{lemma}
\label{lem:transl}
Let $K$ be a number field, and let $f_1,\dots, f_r\in K(X)$ be rational functions 
 of degree $>1$. Set $G:=\mathrm{Gal}(\prod_{i=1}^r (f_i(X)-t)/K(t))$, and for each $i\in \{1,\dots, r\}$, let $U_i$ be the stabilizer of a root of $f_i(X)-t$ in $G$. The following are equivalent:
\begin{itemize}
\item[1)] The set of fake values of $\{f_1,\dots, f_r\}$ is not a thin subset of $K$ (in the sense of Serre)\footnote{Here, a subset $S\subseteq \mathbb{P}^1_K$ is called thin, if it is contained in the union of a finite set and finitely many value sets of nontrivial branched coverings $\varphi_i: X_i\to \mathbb{P}^1_K$ of curves over $K$.}.
\item[2)] All 
elements of $K$ are pseudo-values of $\{f_1,\dots, f_r\}$, i.e., the set $\{f_1,\dots, f_r\}$ locally represents $K$.
\item[3)] For every finite extension $F\supseteq K$, the set $\{f_1,\dots, f_r\}$ locally represents $F$.
\item[4)] Every element of $G$ fixes a point in its action on the roots.
\item[5)] $U_1,\dots, U_r$ yield a normal covering of $G$.
\end{itemize}
\end{lemma}
\begin{proof}
Of course, 3) implies 2). Furthermore, since the union of the sets of $K$-values of finitely many functions of degree $>1$ is thin by definition, 2) implies 1).

Next, we will show 1)$\Rightarrow 4)$. Assume therefore that 4) does not hold, and choose any $t_0\in K$ such that $\mathrm{Gal}(\prod_{i=1}^r (f_i(X)-t_0)/K) = G$. The set of these $t_0$ is the complement of a thin subset by Hilbert's irreducibility theorem. Denote the respective $G$-extension of $K$ by $F/K$. Choose a conjugacy class of elements $x\in G$ fixing no root. There are infinitely many primes unramified in $F/K$ and possessing this class as the Frobenius class in $F/K$, and in particular $f_i(X)=t_0$ then has no $K_p$-rational solution for any $i=1,\dots, r$. This shows that $t_0$ is not a pseudo-value of $\{f_1,\dots, f_r\}$. Hence the set of pseudo-values, and a fortiori the set of fake values, is a thin set. We have thus derived that 1) implies 4).\\
Now, we will show 4)$\Rightarrow $3). Let $F/K$ be a finite extension, and choose $t_0\in F$. We have to show that, for almost all primes $p$ of $F$, $t_0$ is an $F_p$ value of some $f_i$. Assume first that $t_0$ is not a critical value of any $f_i$.  In this case, $\mathrm{Gal}(\prod_{i=1}^r (f_i(X)-t_0)/F)$ is permutation-isomorphic to a subgroup of $\mathrm{Gal}(\prod_{i=1}^r (f_i(X)-t)/F(t))$, which itself is permutation-isomorphic to a subgroup of $G$. Now let $p$ be a prime of $F$  unramified in the splitting field of $\prod_{i=1}^r (f_i(X)-t_0)$. The Frobenius class at $p$ in this splitting field consists of elements of $G$, which by assumption have a fixed point. This implies that at least one factor $f_i(X)-t_0$ has a root in $F_p$, i.e., $t_0$ is an $F_p$-value of some $f_i$.
Next, assume that there exists some $i\in \{1,\dots, r\}$ such that $t_0$ is a critical value of $f_i$. We may assume $\infty\notin f_i^{-1}(t_0)$, since otherwise $t_0$ is even an $F$-value of $f_i$. Then the numerator of $f_i(X)-t_0$ factors 
as $\prod_{j=1}^{d_i}g_{i,j}(X)^{e_{i,j}}$ for suitable (pairwise coprime) irreducible polynomials $g_{i,j}(X)\in F[X]$ ($i=1,\dots, r$). Here, as mentioned in Section \ref{sec:ratfcts}, the $e_{i,j}$ are the ramification indices of places extending $t\mapsto t_0$ in a root field of $f_i(X)-t$.
Furthermore, the roots of a given irreducible $g_{i,j}$ correspond to all the orbits of the inertia group $I$ at (a fixed place extending) $t\mapsto t_0$ joined into a common orbit of the decomposition group $D$ at $t\mapsto t_0$, cf., e.g., \cite[Theorem I.9.1]{MM}. 
This yields an identification of the Galois group $\mathrm{Gal}(\prod_{i=1}^r (f_i(X)-t_0)/F)$ with the quotient group $D/I$. Now let $p$ be any prime of $F$ such that the product of all $g_{i,j}$ remains separable modulo $p$. Then $p$ is unramified in the splitting field of $\prod_{i=1}^r (f_i(X)-t_0)$, and its decomposition group at $p$ is naturally identified with a cyclic subgroup of $D/I$. Let $\tau$ be any generator of an extension of this subgroup to a cyclic subgroup of $D$. Since $\tau$ is assumed to fix a point, its projection to $D/I$ certainly fixes an orbit of $I$, implying that the decomposition group at $p$ fixes a root of at least one $g_{i,j}(X)$. This translates to $t_0$ being an $F_p$-value of the respective rational function $f_i$, completing the implication 4)$\Rightarrow$ 3). 

Finally, the equivalence of 4) and 5) is straightforward upon noting that the orbits of $G$ are exactly the coset spaces $G/U_i$, $i=1,\dots, r$; cf.\  Section \ref{sec:intersect}.\end{proof}

The group-theoretical translation in Lemma \ref{lem:transl} has several immediate, but noteworthy consequences. In particular, the property to locally represent a number field is compatible with composition of functions, thus further emphasizing the special role of indecomposable functions in Question \ref{ques:2}.

\begin{corollary}
\label{cor:transl2}
If $\{f_1,\dots, f_r\}$ and $\{g_1,\dots, g_s\}$ each locally represent $K$, then so does $\{f_i\circ g_j\mid 1\le i\le r; 1\le j\le s\}$. 
\end{corollary}
\begin{proof}
Let $L\supseteq K(t)$ be the splitting field of $\prod_{i,j} (f_i(g_j(X))-t)$, and let $\sigma\in \mathrm{Gal}(L/K(t))$. Since $L$ in particular contains the splitting field of $\prod_i (f_i(X)-t)$ and $\{f_1,\dots, f_r\}$ is locally representing, the equivalence 2)$\Leftrightarrow$4) of Lemma \ref{lem:transl} implies that $\sigma$ fixes a root $y$ of some $f_i(X)-t$; in particular $\sigma\in \mathrm{Gal}(L/K(y))$. But since $g_j(x)=y$ implies $f_i(g_j(x))=t$, $L$ contains the splitting field of $\prod_{j}(g_j(X)-y)$ over $K(y)$; and since $\{g_1,\dots, g_s\}$ is also locally representing, this means that $\sigma$ fixes a root $x$ of some $g_j(X)-y$. Then $x$ is a root of $f_i(g_j(X))-t$, and the assertion follows from the equivalence 2)$\Leftrightarrow$4) in Lemma \ref{lem:transl}.
\end{proof}

Finally, the following observations show that it does not make sense to replace  ``rational functions" by ``polynomials" (see already (4.2) of \cite{Fried74}, which asserts that  a set of polynomials attaining every residue class mod $p$ for almost all primes $p$ of $K$ must contain a linear polynomial), or to restrict to the case $r=1$ in Question \ref{ques:2}.
\begin{corollary}
\label{cor:transl3}
\begin{itemize}
\item[a)]
If $f_1,\dots, f_r\in K[X]$ are polynomials of degree $>1$, then $f_1,\dots, f_r$ do not locally represent $K$.
\item[b)] If $f_1\in K(X)$ is a single rational function of degree $>1$, then $f_1$ does not locally represent $K$.
\end{itemize}
\end{corollary}
\begin{proof}
The inertia group generator at $t\mapsto \infty$ in the joint splitting field of $f_i(X)-t$ ($i=1,\dots r$) acts as a full cycle on each set of roots, and hence in particular is fixed point free. This shows a). In the situation of b), the Galois group acts transitively, and it is then well-known that it must contain a fixed point free element.
\end{proof}

\section{Proof of Theorem \ref{thm:main} and Theorem \ref{thm:sym}} 
\label{sec:mainproofs}
 

To make use of the results of Section \ref{sec:largemon}, involving primitive permutation groups with a unique minimal normal subgroup, we establish a technical lemma. To state it, for a given collection $\mathcal{C}$ of finite simple groups, we call a finite Galois extension $F/K$ a $\mathcal{C}$-extension, if all (Jordan-H\"older) composition factors of its Galois group are contained in $\mathcal{C}$. Note that, since composita of $\mathcal{C}$-extensions are again $\mathcal{C}$-extensions, any finite Galois extension contains a unique maximal $\mathcal{C}$-subextension.

\begin{lemma} \label{nonSolvableMess}
Let $\kappa$ be a field of characteristic zero and $\mathcal{C}$ a collection of finite simple groups.
Let $\Omega_{1},\Omega_{2},\dots, \Omega_{r}$ be finite Galois extensions of $\kappa$ with Galois groups $G_{i}:=\mathrm{Gal}(\Omega_{i}/\kappa)$, $i=1,2,\dots, r$. Assume that every $G_{i}$ has a \textit{unique} minimal normal subgroup $N_{i} \cong H_{i}^{m_{i}}$ for some simple group $H_{i} \not \in \mathcal{C}$, and that all composition factors of $G_i/N_i$ lie in $\mathcal{C}$. Let $\Omega:=\Omega_{1}\Omega_{2}\cdots \Omega_{r}$.
Let $M/\kappa$ be a finite $\mathcal{C}$-extension containing the maximal $\mathcal{C}$-subextension of $\Omega/\kappa$. 
 Then the following holds:
\begin{itemize}
\item[(1)] For all $i\in \{1,\dots, r\}$, it holds that $\mathrm{Gal}(\Omega_{i}M/M) \cong N_{i}$.
\item[(2)] For all $i,j\in \{1,\dots, r\}$, it holds that $\Omega_{i} = \Omega_{j}$ or $\Omega_i M \cap \Omega_j M=M$.
\end{itemize}
If additionally all of $H_1,\dots, H_r$ are nonabelian, then the following holds:
\begin{itemize}
\item[(3)] 
Assume that $i\in \{1,\dots, r\}$ is such that $\Omega_{i} \neq \Omega_{j}$ for all $i\ne j\in \{1,\dots, r\}$, and let $E$ be the compositum of $M$ and all $E_j$, $i\ne j\in \{1,\dots, r\}$. Then the restriction map induces an isomorphism
\[\mathrm{Gal}(\Omega_{i}E/E) \cong \mathrm{Gal}(\Omega_iM/M) \cong N_{i}.
\]
\end{itemize}
\end{lemma}

\begin{proof}
Let us set $E_i:=\Omega_iM$.\\
We first show 1). For $i\in \{1,\dots, r\}$, let $K_i\subset \Omega_i$ be the fixed field of $N_i$. By definition of $M$, we have $K_i\subseteq M$, but $\Omega_i\not\subseteq M$, and hence $M\cap \Omega_i = K_i$, since there is no normal extension of $\kappa$ properly contained between $K_i$ and $\Omega_i$. Thus $\mathrm{Gal}(E_i/M)\cong \mathrm{Gal}(\Omega_i/\Omega_i\cap M) = \mathrm{Gal}(\Omega_i/K_i)=N_i$, showing 1).

We proceed with 2). Similarly as in 1), $\mathrm{Gal}(E_i/E_i\cap E_j) \cong \mathrm{Gal}(\Omega_i/(E_i\cap E_j)\cap \Omega_i)$, and since $E_j\cap E_j$ is a Galois extension of $\kappa$ containing $M$, one has that $(E_i\cap E_j)\cap \Omega_i$ is a Galois intermediate extension of $\Omega_i/\kappa$ containing $K_i$, i.e., equals $K_i$ or $\Omega_i$. In the first case, $\mathrm{Gal}(E_i/E_i\cap E_j) = N_i =\mathrm{Gal}(E_i/M)$, yielding $E_i\cap E_j=M$, while in the second case clearly $E_i=E_i\cap E_j \subseteq E_j$. By symmetry, we may thus from now on assume
that $E_{i}=E_{j}$. 

Let $\Omega_{ij}:=\Omega_{i}\Omega_{j}$, then $\Omega_{i}M=\Omega_{j}M=\Omega_{ij}M$. Notice that under this assumption we have
\begin{equation}
	\mathrm{Gal}(\Omega_{ij}M/\Omega_{i}) = \mathrm{Gal}(\Omega_{i}M/\Omega_{i}) \cong \mathrm{Gal}(M/\Omega_{i} \cap M). \label{smallFactors}
\end{equation}
On the other hand, we also have an epimorphism
\begin{equation}
	\mathrm{Gal}(\Omega_{ij}M/\Omega_{i}) \twoheadrightarrow \mathrm{Gal}(\Omega_{ij}/\Omega_{i}) \cong \mathrm{Gal}(\Omega_{j}/\Omega_{i} \cap \Omega_{j}). \label{largeFactors}
\end{equation}
From (\ref{smallFactors}) we infer that all composition factors of $\mathrm{Gal}(\Omega_{ij}M/\Omega_{i})$ belong to $\mathcal{C}$. Thus, from the projection in (\ref{largeFactors}) the same holds for the composition factors of $\mathrm{Gal}(\Omega_{j}/\Omega_{i}\cap \Omega_{j}) \trianglelefteq G_j$. But by assumption, every nontrivial normal subgroup of $G_j$ has a composition factor outside of $\mathcal{C}$. It follows that $\Omega_{i}\cap \Omega_{j}=\Omega_j$, i.e., $\Omega_j\subseteq \Omega_i$. 
Likewise it follows that $\Omega_{i} \subseteq \Omega_{j}$, hence the claim.

Next, we prove 3). 
For notational convenience, we assume $i=1$. We may and will furthermore assume that $\Omega_{j}\ne \Omega_{k}$ for all $1\le j\ne k\le r$ (or otherwise, simply delete some of the $\Omega_{j}$, $2\le j\le r$). 
Let $\Gamma:=\mathrm{Gal}(\Omega_{1}\Omega_{2}\cdots \Omega_{r}M/M) = \mathrm{Gal}(E_{1}E/M)$. 
Via the restrictions $\Gamma \to \mathrm{Gal}(E_{j}/M) = N_{j}$ ($j=1,\dots, r$), $\Gamma$ embeds as a subgroup into $N_{1}\times\dots\times N_{r}$. We claim that this embedding is in fact an isomorphism.
From this, the assertion follows, since $\mathrm{Gal}(\Omega_{1}E/E)\le \Gamma$ is exactly the subgroup acting trivially on each $E_{j}$, $j=2,\dots, r$; i.e., the preimage of $N_{1}\times\{1\}\times\dots\times\{1\}$ under the above embedding. 

To prove the claim, we may proceed by induction over $r$, with the base case $r=1$ simply being assertion 1) of this theorem.
So assume inductively that $\mathrm{Gal}(E/M)=\mathrm{Gal}(E_{2}\cdots E_{r}/M)\cong N_{2}\times\dots\times N_{r}$. To derive that $\mathrm{Gal}(E_{1}E/M) \cong \mathrm{Gal}(E_{1}/M)\times \mathrm{Gal}(E/M)$ as claimed, it remains to show that $E$ and $E_{1}$ are linearly disjoint over $M$, i.e., that $E\cap E_{1} = M$. 

Assume otherwise, then $\mathrm{Gal}(E_{1}/E_{1}\cap E)\cong \mathrm{Gal}(E\Omega_{1}/E) \cong\mathrm{Gal}(\Omega_{1}/\Omega_{1}\cap E) \triangleleft G_{1}$ must be a normal subgroup properly contained in $N_{1} = \mathrm{Gal}(\Omega_{1}/K_{1})\cong\mathrm{Gal}(E_{1}/M)$, and hence must be trivial. In other words, $\Omega_{1}\subseteq E$. Therefore, $\mathrm{Gal}(E_{1}/M)$ is a quotient of $\mathrm{Gal}(E/M)$. Since the latter is assumed to be a direct product of {\it nonabelian} simple groups, the only normal subgroups are the kernels of the projections to a subset of the simple direct factors. I.e., $\mathrm{Gal}(E/E_{1})$ would have trivial image on at least one of these factors. Let $j\in \{2,\dots, r\}$ be the index such that this factor belongs to $N_{j}$. In particular, $\mathrm{Gal}(E/E_{1})$ projects to a proper subgroup of $\mathrm{Gal}(E_{j}/M)=N_{j}$. But linear disjointness of $E_{1}$ and $E_{j}$ over $M$, as ensured by assertion 2), means that $\mathrm{Gal}(E_{j}E_{1}/E_{1})$, and hence in particular the image of $\mathrm{Gal}(E/E_{1})$ under restriction to $E_{j}$ is all of $N_{j}$, a contradiction. This concludes the proof.
\end{proof}

From Lemma \ref{nonSolvableMess}, we obtain the following, which translates a ``normal covering" property in a compositum of several Galois extensions (with certain technical assumptions) into a slightly weaker covering property in an individual extension.
In view of Lemma \ref{lem:transl}, this will in particular lead to a strong necessary condition for a set of functions $\{ f_{1},f_{2}, \dots, f_{r} \}$ to locally represent $\mathbb{Q}$ when there is a function $f_{i}$ of large degree and with non-solvable monodromy group.


\begin{corollary} \label{cosetFixedpoint}
Let $\mathcal{C}$ be a collection of simple groups. Let $\kappa$ be a field of characteristic zero and let $p_{1},p_{2}, \dots p_{l},p_{l+1}, \dots, p_{r}\in \kappa[X]$ (with $0\le l < r$) be polynomials of degree $\ge 2$.
For $i=1,\dots, r$, let $\Omega_{i}$ be the splitting field of $p_i$ over $\kappa$ and set $G_i:=\mathrm{Gal}(\Omega_i/\kappa)$. Finally, let $\Omega$ be the compositum of all $\Omega_{i}$, $i\in\{1,\dots, r\}$, put $G:=\mathrm{Gal}(\Omega/\kappa)$, and let $U_i< G$ be the stabilizer of a root of $p_i$. Assume that all of the following hold.
\begin{itemize}
\item[i)] for $i \leq l$, all composition factors of $G_i$ belong to $\mathcal{C}$,
\item[ii)] for $l+1 \leq i \leq r$ it holds that $G_i$ 
has a unique minimal normal subgroup $N_{i} \cong H_{i}^{m_{i}}$, where $H_{i}$ is non-abelian simple and does not belong to $\mathcal{C}$, and furthermore all composition factors of $G_i/N_i$ belong to $\mathcal{C}$.
\item[iii)] $\{U_1,\dots, U_r\}$ yields a normal covering of $G$, but no proper subset does so.
\end{itemize}
Then for every index $l+1 \leq i_{0} \leq r$, there is a coset $\sigma N_{i_{0}}$ of $G_{i_0}$ such that every element of $\sigma N_{i_{0}}$ fixes a root of some $p_j$ with splitting field $\Omega_{j}$ satisfying $\Omega_{j}=\Omega_{i_{0}}$.
%
\end{corollary}
\begin{proof}
Let $i_{0}\ge l+1$ be given and let
\[
\mathcal{S}_{i_{0}}:=\{ j \mid \Omega_{j}=\Omega_{i_{0}} \}.
\]
Furthermore, let
\[
\Omega_{i_{0}}^{c}:=\prod_{j \in \{l+1,l+2,\dots, r  \} \setminus \mathcal{S}_{i_{0}}} \Omega_{j}.
\]
Since the set $\{U_i\mid i\in \{1,\dots, r\}\setminus\mathcal{S}_{i_0}\}$ does not yield a normal covering of $G$, there exists an element $\sigma \in \mathrm{Gal}(\Omega/\kappa)$ that does not fix a root of any $p_j$ with $j\in \{1,\dots, r\} \setminus \mathcal{S}_{i_{0}}$. Let $M \subset \Omega$ be the maximal $\mathcal{C}$-subextension of $\Omega/\kappa$.
Then, a fortiori, any element of $\sigma \mathrm{Gal}(\Omega/ \Omega_{i_{0}}^{c}M)$ fixes a root of some $p_j$ where $j \in \mathcal{S}_{i_{0}}$.

Let $E$ be the compositum of $M$ and $\Omega_{i_{0}}^{c}$, i.e. the compositum of $M$ and all $\Omega_{j} \neq \Omega_{i_{0}}$. By Lemma \ref{nonSolvableMess}(3) (applied with indices $i:=i_0$ and $j\in \{l+1,\dots r\}\setminus\mathcal{S}_{i_0}$), 
it follows
that the restriction map yields an isomorphism
\[
\mathrm{Gal}(\Omega/\Omega_{i_{0}}^{c}M)=\mathrm{Gal}(\Omega_{i_{0}}E/E)\cong \mathrm{Gal}(\Omega_{i_{0}}M/M) \cong N_{i_{0}}.
\]
Thus, let $\tau \in N_{i_{0}}$ and let $\tilde{\tau} \in \mathrm{Gal}(\Omega / \Omega_{i_{0}}^{c}M)$ be such that $\tilde{\tau}|_{\Omega_{i_{0}}}=\tau$. Then $\sigma \tilde{\tau}$ fixes a root of $p_j$ for some $j \in \mathcal{S}_{i_{0}}$, and since
\[
    (\sigma \tilde{\tau})|_{\Omega_{i_{0}}} =\sigma|_{\Omega_{i_{0}}} \tau, 
\]
it follows that $\sigma|_{\Omega_{i_{0}}} \tau$ fixes a root of some $p_j$ with $j \in \mathcal{S}_{i_{0}}$. Thus, any element of the coset $\sigma|_{\Omega_{i_{0}}}N_{i_{0}}$ fixes a root of some $p_j$ with $j \in \mathcal{S}_{i_{0}}$.
\end{proof}

\begin{proof}[Proof of Theorem \ref{thm:main}]
Fix a positive integer $r$, and let $\{f_{1},f_{2},\dots, f_{r} \} \subset K(X)$ be a minimal locally representing set of geometrically indecomposable functions with monodromy groups $G_i:=\mathrm{Mon}(f_{i})$. 
For $i=1,2,\dots, r$, let $x_{i}$ be a root of $f_{i}-t$ and let $\Omega_{i}$ be the splitting field of $f_{i}-t$ over $K(t)$, so that $G_{i}=\mathrm{Mon}(f_{i})=\mathrm{Gal}(\Omega_{i}/K(t))$. Let $\Omega=\Omega_{1}\Omega_{2}\cdots \Omega_{r}$, $G=\mathrm{Gal}(\Omega/K(t))$, and let $U_{i}<G$ be the stabilizer of $x_i$. Since $\{f_{1},f_{2},\dots, f_{r} \}$ is a minimal set locally representing $K$, it follows by Lemma \ref{lem:transl} that $\mathcal{U}:=\{U_{1},U_{2},\dots, U_{r} \}$ yields a normal covering of $G$, but none of its proper subsets does. 

Assuming that $G_r$ is nonsolvable and $\deg(f_{r})$ is sufficiently large, Theorem \ref{thm:OldClassification} and Remark \ref{rem:largemon} yield an embedding (of abstract groups)
\[
A_{m}^{t} \leq G_{r} \leq S_{m} \wr S_{t},~t \leq 8,
\]
where $A_{m}$ does not belong to the finite set $\mathcal{S}$ of Remark \ref{rem:largemon}. 
The latter property implies that, up to a reordering of the indices, we may assume the existence of some $1\le l\le r-1$ such that the following hold.

\begin{itemize}
    \item For $i\le l$, the group $G_i$ has no composition factor $A_m$. 
    \item For $l+1\le i\le r$, the group $G_i \le S_m\wr S_{t_i}$ acts in a product action (with $1\le t_i\le 8$).\footnote{I.e., a faithful primitive action arising from the action of $S_m\wr S_{t_i}$ (or possibly, of $A_m\wr S_{t_i}$, in case $G_i$ embeds into that group) on cosets of a suitable subgroup $V_i\wr S_{t_i}$ where $V_i< S_m$.} In particular, $G_i$ has a unique minimal normal subgroup $N_i\cong A_m^{t_i}$, and $G_i/N_i$ has no composition factor $A_m$.
\end{itemize}

By applying Corollary \ref{cosetFixedpoint} with $p_{i}(X)=f_{i}-t$ and with $\mathcal{C}$ the set of all finite simple groups other than $A_{m}$, it follows that there is a coset $\sigma N_{r} \subset G_{r}$ such that every element $\tau \in \sigma N_{r}$ fixes a root of some $f_{j}-t$ with splitting field $\Omega_{j}=\Omega_{r}$.  Since each of the actions on roots is the restriction of some product action of the abstract group $S_m\wr S_{t}$  (or of $A_m\wr S_{t}$, in case $G_r$ embeds into that group) to $G_r$, there are a total of $|\{ j \mid \Omega_{j}=\Omega_{r} \}| \le r$ such actions of $S_m\wr S_{t}$ (resp., of $A_m\wr S_{t}$)  
such that every element of $\sigma N_r$ has a fixed point in at least one of them. Lemma \ref{lem:ProductFixedPointFree} then implies the existence of $\le r$ proper subgroups of $A_{m}$ or of $S_{m}$ whose conjugates cover a full coset of $A_{m}$. By Proposition \ref{prop:groupcover}, this implies $m\le c^{-1}r$ with an absolute constant $c>0$. Thus, $|\Mon(f_r)|$, and with it $\deg(f_r)$, is after all bounded in terms of $r$. This completes the proof.
\end{proof}

\begin{proof}[Proof of Theorem \ref{thm:main}']
We proceed in the same way as in the proof of Theorem \ref{thm:main}. Let $N$ be a constant such that Theorem \ref{NZTheo} applies to rational functions of degree at least $\sqrt{N}$.
Let $\{f_{1},f_{2},\dots, f_{r}\} \in K(X)$ be a minimally locally representing set of geometrically indecomposable rational functions with monodromy groups $G_{i}:=\mathrm{Mon}(f_{i})$. Assume for the sake of a contradiction that $G_{r}$ is non-solvable and $\deg(f_{r}) \geq N$. Denote again by $x_{i}$ a root of $f_{i}-t$, and let $\Omega_{i}$ be the splitting field of $f_{i}-t$ over $K(t)$. By Theorem \ref{NZTheo}, $G_{r}$ has a unique minimal subgroup $N_{r} \cong A_{m}^{t}$ where $t \in \{1,2\}$. Arguing as in the proof of Theorem \ref{thm:main}, we see that there is a coset $\sigma N_{r} \subset G_{r}$ such that every $\tau \in \sigma N_{r}$ fixes a root of some $f_{j}-t$ with splitting field $\Omega_{j}=\Omega_{r}$.
However, we claim that this cannot hold. By Theorem \ref{NZTheo} there are two cases to consider:
\begin{itemize}
\item[Case 1:] $G_{r}=A_{m}$ or $S_{m}$ for some $m \geq \sqrt{N}$. In this case any function $f_{j}$ with $\Omega_{j}=\Omega_{r}$ is of degree $m$ or $m(m-1)/2$. Since we may as well assume $\sqrt{N} \geq 6$, it follows from \cite[Theorem 5.2.A]{DM} that all subgroups of $G_{r}$ of index $m$ are conjugate and that all subgroups of index $m(m-1)/2$ are conjugate. Thus, it follows that a coset of $A_{m}$ in $A_m$ or $S_{m}$ is covered by the conjugates of two proper subgroups, which by Proposition \ref{prop:groupcover} is impossible for sufficiently large values of $N$.
\item[Case 2:]
$A_{m}^{2} \leq G_{r} \leq S_{m}^{2} \rtimes S_{2}$ acting as a product type group, and $\deg(f_{r})=m^{2}\ge N$. 
Theorem \ref{NZTheo} ensures that all the actions of $G_r$ on roots of $f_j-t$, where $j$ is such that $\Omega_j=\Omega_r$, are permutation-isomorphic, and hence have point stabilizers conjugate under $\mathrm{Aut}(G_r)$. By \cite[Lemma 1.1]{Rose}, $\mathrm{Aut}(G_r)$ embeds into $\mathrm{Aut}(A_m^2) = \mathrm{Aut}(A_m)\wr C_2$. But since $m>6$, one has $\mathrm{Aut}(A_m)=S_m$, and hence $\mathrm{Aut}(G_r)\le S_m\wr S_2$. 
I.e., $\mathrm{Aut}(G_r)$ is no larger than the symmetric normalizer $N_{S_{m^2}}(G_r)$. It then follows from \cite[Theorem 4.2B]{DM} that all $\mathrm{Aut}(G_r)$-conjugates of $U$ are actually conjugate inside $G_r$, i.e., 
there is only one possible conjugacy class of subgroups for the point stabilizers $\mathrm{Gal}(\Omega_{r}/K(x_{j}))<G_{r}$, where $\Omega_{j}=\Omega_{r}$. 
It then again follows from Lemma \ref{lem:ProductFixedPointFree} that every coset of $N_{r}$ in $G_{r}$ has an element which does not fix a root of any $f_{j}-t$ with $\Omega_{j}=\Omega_{r}$.
\end{itemize}
Thus, in both cases every coset of $N_{r}$ in $G_{r}$ contains an element which is fixed point free in the action on the roots of all $f_{j}(X)-t$ with splitting field $\Omega_{r}$, contradicting that $\{ f_{1},f_{2}, \dots, f_{r} \}$ locally represents $K$. 
\end{proof}

\begin{remark}
The proof of Theorem 1.1' can be carried out analogously under the following weaker assumption, which is an immediate corollary of Theorem \ref{NZTheo}: the number of faithful primitive genus-$0$ actions (i.e., permutation actions in which a suitable tuple of elements becomes a genus-$0$ tuple) of a finite group $G$ is absolutely bounded from above. While, to our knowledge, such a result is not available in the previous literature, its proof should require less intricate means than the full proof of Theorem \ref{NZTheo}. In the special case where the number of branch points of each rational function arising from such a genus-$0$ action is at least $5$, the result becomes an immediate corollary of the main theorem of \cite{GS}.
\end{remark}

We now turn to the proof of Theorem \ref{thm:sym}. The proof uses the same idea as that of Theorem \ref{thm:main}, namely that the Galois group of the product of the $f_{j}(X)-t$ of large degree factors as a direct product.
\begin{proof}[Proof of Theorem \ref{thm:sym}] 
Assume that $\{ f_{1},f_{2}, \dots, f_{l},f_{l+1}, \dots, f_{r} \} \in K(X)$ is a minimally locally representing set of functions, all with alternating or symmetric monodromy group in the standard action, and such that
\begin{itemize}
\item $\deg(f_{i}) \in \{2,3,4,6\}$ for $i \leq l$
\item $\deg(f_{i}) \in \{ 5,7,8, \dots \}$ for $l+1 \leq i \leq r$. 
\end{itemize}
We apply Corollary \ref{cosetFixedpoint} with $\mathcal{C}=\{C_2, C_3, A_6\}$, $\kappa=K(t)$ and $p_i(X)=f_i(X)-t$, $i=1,\dots, r$. 
Let $\Omega_{i}=\mathrm{Split}(f_{i}(X)-t/K(t))$. For $i \leq l$ the composition factors of $\mathrm{Mon}(f_{i})$ are among $C_2$, $C_3$ or $A_{6}$, while for $i \geq l+1$ it holds that $\mathrm{Mon}(f_{i})$ is non-solvable with unique minimal subgroup $A_{\deg(f_{i})}\notin \mathcal{C}$. It follows by Corollary \ref{cosetFixedpoint} that for every index $i_{0}$ with $l+1 \leq i_{0} \leq r$ there is a coset of $A_{\deg(f_{i_{0}})}$ in $\mathrm{Mon}(f_{i_{0}})$ in which every element fixes a root of one of the functions $f_{j}$ satisfying $\Omega_{j}=\Omega_{i_{0}}$. Since $\deg(f_{i}) \neq 6$ for $l+1 \leq i \leq r$ it follows from \cite{DM}, Theorem 5.2A and Theorem 5.2B, that the actions of $\mathrm{Gal}(\Omega_{i_{0}}/K(t))$ on the the roots of $f_{j}(X)-t$ with $\Omega_{j}=\Omega_{i_{0}}$ are all equivalent to the (restriction of the) transitive action of $S_{\deg(f_{i_0})}$ on $\deg(f_{i_0})$ points. It is now easy to see that for $n \geq 5$, both cosets of $A_{n}$ in $S_{n}$ have a fixed point free element, and hence by Corollary \ref{cosetFixedpoint} the set $\{f_{1},f_{2},\dots, f_{l}\}$ is already locally representing, a contradiction.
\end{proof}

The exact same argument, with $\kappa = K$ a number field instead of the function field $K(t)$, yields the following result on intersective polynomials whose factors have ``generic" Galois group.
\begin{theorem}
\label{thm:sym2}
Let $K$ be a number field, and $f=f_1\cdots f_r\in K[X]$ be a (nonlinear) newly intersective polynomial such that $\mathrm{Gal}(f_i) \in \{S_{\deg(f_i)}, A_{\deg(f_i)}\}$ for all $i\in \{1,\dots, r\}$. Then each $f_i$ has degree one of $2$, $3$, $4$ and $6$.
\end{theorem}

\section{Existence results}
\label{sec:exist}
In this section, we will give concrete examples to demonstrate that those classes of rational functions not excluded by Theorems \ref{thm:main} and \ref{thm:sym} do indeed occur in suitable locally representing sets. 
We will  ensure to always obtain examples over $\mathbb{Q}$. We begin somewhat informally in Section \ref{sec:sporadic} by listing some noteworthy low-degree nonsolvable examples. In Section \ref{sec:symmetric}, we specifically address low degree functions with symmetric monodromy group. In Sections \ref{sec:genus0} and \ref{sec:genus1}, we deal more systematically with functions as in Case 2) of Theorem \ref{thm:main}, thus eventually deducing Theorem \ref{thm:converse}.
 
\subsection{Some ``sporadic" low-degree examples}
\label{sec:sporadic}
\begin{example}
\label{ex:sporadic}
\begin{itemize}
\item[a)] Let 
$f_1(X) =\frac{(X^2+10X+5)^3}{X}$,   $f_2(X) =X^3(X^2+5X+40)$, and $f_3(X)= \frac{(5X)^3(8X^2+25X+20)^3}{(X^2+5X+5)^5}$. Then $f_1(X)-t$, $f_2(X)-t$ and $f_3(X)-t$ all have geometric monodromy group $A_5$ (acting primitively on $6$, $5$ and $10$ points respectively) and share the same Galois closure over $\mathbb{Q}(t)$; indeed they all parameterize subcovers of Klein's icosahedral cover $\mathbb{P}^1\to \mathbb{P}^1$, i.e., the genus-$0$ Galois cover of degree $60$ with three branch points and ramification indices $2$, $3$ and $5$, cf.\ Section \ref{sec:gal_genus0}. Computation of such low degree rational functions with prescribed ramification data is standard nowadays; for more details, we refer to comparable calculations carried out explicitly in \cite[Chapter I.9]{MM}. Note that the functions $f_1, f_2, f_3$ appear, up to linear transformations, as Cases Vb), c) and d) in \cite[Theorem 1.1]{Pak}. In order to ensure that the computed functions indeed have the same Galois closure, it suffices to note that the icosahedral cover is unique up to equivalence induced by fractional linear transformations, whence it suffices to match the branch points of the three functions (with our choice, the points $0$, $1728$ and $\infty$). 
Note furthermore that these functions have geometric monodromy group $A_5$, but arithmetic monodromy group $S_5$ over $\mathbb{Q}$, as can be deduced directly from Proposition \ref{prop:bcl}). Indeed, due to $f_1$ having a unique branch point of ramification index $5$, the arithmetic monodromy group must contain the full symmetric normalizer $C_5\rtimes C_4$ of $C_5\le S_5$. In fact, the exact constant extension inside the Galois closure is $\mathbb{Q}(\sqrt{5})/\mathbb{Q}$, the quadratic subextension of $\mathbb{Q}(\zeta_5)/\mathbb{Q}$ (see, e.g., \cite[Section 8.3.2]{Serre}. 
Since the $S_5$-conjugates of the maximal subgroups of index $10$ and $6$ cover the whole group, $\{f_1,f_3\}$ is an example of a (minimal) set locally representing $\mathbb{Q}$. Furthermore, the conjugates of the index-$6$ and index-$5$ subgroups, while not covering all of $S_5$, do cover all of $A_5$, whence every automorphism fixing $\mathrm{Fix}(A_5)= \mathbb{Q}(\sqrt{5})(t)$ also fixes a root of $f_1-t$ or of $f_2-t$. Complement this by the function $f_4(X)=X^5$; the splitting field of $f_4-t$ contains $\sqrt{5}$ as well, and one has $\mathrm{Gal}(X^5-t/\mathbb{Q}(t))=C_5\rtimes C_4 \le S_5$ and $\mathrm{Gal}(X^5-t/\mathbb{Q}(\sqrt{5})(t))=D_5$. In particular,  every automorphism which does not fix $\sqrt{5}$ is of even order in $\mathrm{Gal}(X^5-t/\mathbb{Q}(t)) \le S_5$ and hence fixes a root of $X^5-t$. In total, $\{f_1,f_2,X^5\}$ is another minimal set of functions locally representing $\mathbb{Q}$.
\item[b)] 
Let $f_{1}(X)=\frac{f_{1,1}(X)}{f_{1,2}(X)}$ and $f_2(X)=\frac{f_{2,1}(X)}{f_{2,2}(X)}$, where
$$f_{1,1}(X):=(77X^3+10989X^2+129816X+496368)^3(77X^2+2376X+15472),$$
$$f_{1,2}(X):=(11X^2-1296)^4(11X^2+143X+621),$$
$$f_{2,1}(X):=7^3\cdot 5^{10}\cdot (11X^2+66X+91)^3(44X^3-33X^2-570X+951),$$
$$f_{2,2}(X):=11\cdot (11X^2+22X-89)^4(22X^2+110X+113)^2.$$
These are two rational functions of degree $11$ and $12$ such that $f_1(X)-t$ and $f_2(X)-t$ have the same Galois closure with Galois group $M_{11}$, the smallest Mathieu group; see \cite{KN18}. Since every element of $M_{11}$ lies in a point stabilizer in one of these two actions, these two rational functions yield another example.
\item[c)] 
Let $$f_1(X) = \frac{(X^3 + 16X^2 + 160X + 384)^3 }{X^2 + 13X + 128}$$ and $$f_2(X) = \frac{(X^9 +11X^8 +4X^7 -868X^6 +6174X^5 -43974X^4 +37492X^3 -28852X^2 -2967X+211)^3(X- 5)}{-2^7 \cdot (X^3 -X^2 -9X+1)^7}.$$ This is taken from \cite[Theorem 3.4]{Koe18} and corresponds to two genus-$0$ actions of $P\Gamma L_2(8)$ of degrees $9$ and $28$.
\end{itemize}
\end{example}

\subsection{Functions with symmetric monodromy group}
\label{sec:symmetric}
We give examples of rational functions of degrees $2,3,4$ and $6$ which are part of a set of rational functions with $S_n$-monodromy and forming a minimal set in the sense of Question \ref{ques:2}. 
One such example is the example of three quadratic functions given already in the introduction. We do not aim at classifying all further such sets and content ourselves with examples showing together that all the degrees do in fact occur.

\begin{lemma}
\label{lem:symmetric}
The following are examples of minimal sets of rational functions with monodromy group $S_n$ locally representing $\mathbb{Q}$.
\begin{itemize}
\item[a)] $f_1(X) = X^4+aX^2+bX$ and $f_2(X) = -\frac{X^3+2aX^2 + a^2X-b^2}{4X}$, for any $a,b\in \mathbb{Q}$ such that $f_1(X)-t$ has Galois group $S_4$.
\item[b)]  $f_1(X) = X^5(5X-6)$, $f_2(X) = \frac{6^6 X}{(X-1)^3(X-16)^2(X-25)}$, $f_3(X) =5X^2-1$.
\end{itemize}
\end{lemma}
\begin{proof}
In a), $f_1(X)-t$ is a generic quartic polynomial (with constant coefficient taken as $t$) and $f_2$ its cubic resolvent. In other words, $\mathrm{Gal}((f_1(X)-t)(f_2(X)-t)/\mathbb{Q}(t))$ is $S_4$, and the subgroups fixing a root of $f_1$ resp., of $f_2$ are $S_3$ and $D_4\le S_4$. Since the conjugates of these two subgroups cover all of $S_4$, the assertion follows.\\
In b), $f_1(X)=t$ gives a cover with Galois group $S_6$, ramified at $t=0$,  $t=-1$ and $t=\infty$, with inertia groups generated by a $5$-cycle, a transposition and a $6$-cycle respectively. We use the fact that $S_6$ has two different conjugacy classes of subgroups $U, V$ both isomorphic to $S_5$, and that the conjugates of $U$ and $V$, together with $A_6$, cover all of $S_6$. It then remains to quickly verify that the fixed fields of all these subgroups in the splitting field of $f_1(X)-t$ are rational function fields; for the subgroup $A_6$, this is easy and follows quickly from the shape of the discriminant of $f_1(X)-t$. For the second class of $S_5$-subgroups, one verifies that the inertia group generators above become a $5$-cycle, triple transposition and an element of cycle structure $(3.2.1)$ in this coset action, yielding a genus-$0$ tuple by Riemann-Hurwitz, and it is then again straightforward to compute that the underlying rational function is given by $f_2$.\end{proof}

\subsection{Examples involving monomials and Chebyshev polynomials}
\label{sec:genus0}
Rational functions $f(X)\in \mathbb{C}(X)$ such that the splitting field of $f(X)-t$ has genus $0$ are well-understood since the days of Klein. 
In particular, the only infinite series of  indecomposable such functions are those linearly related over $\mathbb{C}$ to monomials $X^p$ (for a prime $p$) or to Chebyshev polynomials $T_p$ (for an odd prime $p$), see Proposition \ref{prop:gal_genus0} and Remark \ref{rem:gal_genus0}. We show how to realize these as part of a pair or triple of rational functions locally representing $\mathbb{Q}$.

\begin{lemma}
\label{lem:cheb}
Let $p$ be an odd prime, $f_1(X)=X^p$ and $f_2(X)=T_p(X)$. There exists a quadratic rational function $f_3(X) \in \mathbb{Q}(X)$ such that $\{f_1,f_2,f_3\}$ is a minimal set locally representing $\mathbb{Q}$. 
\end{lemma}
\begin{proof}
Both $f_1$ and $f_2$ have monodromy group $AGL_1(p) = C_p\rtimes C_{p-1}$ over $\mathbb{Q}$. In particular, $\mathrm{Gal}((f_1(X)-t)(f_2(X)-t)/\mathbb{Q}(t))$ is a subgroup of $AGL_1(p)^2$. Furthermore, the splitting fields of both $f_1(X)-t$ and $f_2(X)-t$ contain the degree $\frac{p-1}{2}$ extension $\mathbb{Q}(\cos(2\pi/p))/\mathbb{Q}$. Indeed, the full constant fields of these splitting fields are $\mathbb{Q}(\zeta_p)$ and $\mathbb{Q}(\cos(2\pi/p))$, respectively, see, e.g., \cite[Proposition 5.4]{GMS}. 
Next, consider the two splitting fields of $f_1(X)-t$ and of $f_2(X)-t$ over the common constant extension $\mathbb{Q}(\cos(2\pi/p))(t)$. Both have Galois group $D_p$, and must furthermore be linearly disjoint, since otherwise their intersection were a normal extension of $\mathbb{Q}(\cos(2\pi/p))(t)$, and would thus contain the quadratic subextension of each $D_p$-extension, due to the fact that $C_p$ is the only nontrivial normal subgroup of $D_p$. But this is impossible because for $f_1$, this quadratic extension is still a constant extension $\mathbb{Q}(\zeta_p)(t)/\mathbb{Q}(\cos(2\pi/p))(t)$, whereas for $f_2$ it is non-constant (ramified at $t=\pm 1$). In total, we have obtained that  $G:=\mathrm{Gal}((f_1(X)-t)(f_2(X)-t)/\mathbb{Q}(t)) \cong \{(g,h)\in AGL_1(p)^2\mid gh^{-1}\in D_p\}$. 

Now set $U:=\{(g,h)\in G\mid gh^{-1}\in C_p\}$. This is an index-$2$ normal subgroup of $G$. Let $F\supset \mathbb{Q}(t)$ be the fixed field of $U$ inside the splitting field of $(f_1-t)(f_2-t)$. Since the splitting field of $f_2-t$ is ramified at $t=\pm 1$ of ramification index $2$, whereas the splitting field of $f_1-t$ is unramified at these points, the inertia group generators at $t=\pm 1$ in the joint Galois closure are of the form $(1,\sigma)\in G$ with $\sigma\in D_p\setminus C_p$. In particular, these are not contained in $U$, whence the quadratic extension $F/\mathbb{Q}(t)$ is ramified at $t=\pm 1$ (and only there). $F$ is thus of the form $F=\mathbb{Q}(\sqrt{c\frac{t+1}{t-1}})$ for some constant $c$, and hence contains the roots of $f_3(X)-t$ with $f_3(X):=\frac{X^2+c}{X^2-c}$.
By construction, an element of $G$ not fixing a root of $f_3-t$ must be of the form $(g,h)$ where $g$ and $h$ are at least not both in $C_p$. But then one of them is an element of $AGL_1(p)\le S_p$ fixing a point, i.e., a root of $f_i-t$ for some $i\in \{1,2\}$.

 Finally minimality follows because an element $(g,h)\in G\le AGL_1(p)^2$ where both $g$ and $h$ are $p$-cycles fixes no root of $f_i-t$ ($i=1,2$), and an element of the form $(g,h)$ where $g$ is a $p$-cycle and $h$ is an involution in $D_p$ fixes no root of $f_1-t$ or $f_3-t$ (and same for $f_2$ and $f_3$).
%
%
\end{proof}

The main underlying idea of the above proof, namely, to construct locally representing sets of rational functions via assembling ``suitable" quadratic subextensions of the splitting fields, can be generalized to provide further-reaching conclusions. We will use it in the next section, and in the following lemma, which constructs examples showing that a minimal set of rational functions locally representing $\mathbb{Q}$ can be of arbitrarily large cardinality.
\begin{lemma}
\label{lem:many}
Let $p\equiv 3$ mod $4$ be a prime. 
Then there exist rational functions $f_1(X),\dots f_r(X) \in \mathbb{Q}(X)$, all linearly related to $X^p$ over $\mathbb{C}$, and $f_{r+1}(X)\in \mathbb{Q}(X)$ of degree $2$, such that $\{f_1,\dots, f_r, f_{r+1}, T_p(X)\}$ is a minimal set locally representing $\mathbb{Q}$.
\end{lemma}

To prove Lemma \ref{lem:many}, we use the following construction to obtain many rational functions, all linearly related to $X^p$ over $\mathbb{C}$, but not over $\mathbb{Q}$.
\begin{lemma}
\label{lem:twist}
Let $r\in \mathbb{N}$, let $K_1,\dots, K_r\supset \mathbb{Q}$ be pairwise linearly disjoint quadratic number fields, and let $p\equiv 3$ mod $4$ be a prime. Then there exists rational functions $f_i\in \mathbb{Q}(X)$, $i=1,\dots, r$, all linearly related to $X^p$ over $\mathbb{C}$, such that the following hold.
\begin{itemize}
\item[i)] $\mathrm{Split}(f_i(X)-t/\mathbb{Q}(t))$ contains $K_i$, for all $i=1,\dots, r$.
\item[ii)] $\mathrm{Gal}(\prod_{i=1}^r (f_i(X)-t)/\overline{\mathbb{Q}}(t)) \cong C_p^r$.
\end{itemize}
\end{lemma}
\begin{proof}
Begin with $g(X) = \frac{T_p(X)-1}{T_p(X)+1}$, i.e., $g$ equals the $p$-th Chebyshev polynomial up to fractional linear transformation. In particular, $\Mon(g) = AGL_1(p)$ and $\Mon_{\overline{\mathbb{Q}}}(g) = D_p$. The transformations have been applied only for convenience to move the critical values $\pm 1$ of $T_p$ to the critical values $0$ and $\infty$ of $g$. In particular, $g(X)=s$ is ramified over $s=0$ and $s=\infty$ with inertia group generated by an involution in $D_p$, and has exactly one more branch point ($s=1$, with inertia group generated by a $p$-cycle). 
Since $p\equiv 3$ mod $4$, the quotient $AGL_1(p)/D_p$ is of odd order, and hence the unique quadratic subextension $F/\mathbb{Q}(s)$ of $\mathrm{Split}(g(X)-s/\mathbb{Q}(s))$ is nonconstant. It is then necessarily ramified at $s=0$ and $s=\infty$, i.e., $F=\mathbb{Q}(\sqrt{cs})$ for some constant $c\ne 0$. Now, to show i), let $K_i=\mathbb{Q}(\sqrt{d_i})$ and set $s=cd_it^2$.  
Consider the extension $\mathbb{Q}(t,x)/\mathbb{Q}(t)$, where $x$ denotes a root of $g(X)-s$.
Since the quadratic extension $\mathbb{Q}(t)/\mathbb{Q}(s) = \mathbb{Q}(\sqrt{cd_is})/\mathbb{Q}(s)$ is ramified at $s=0$ and $s=\infty$, Abhyankar's lemma (e.g., \cite[Theorem 3.9.1]{St}), applied to the compositum of $\mathbb{Q}(t)$ and $\mathbb{Q}(x)$ over $\mathbb{Q}(s)$, 
 yields that $\mathbb{Q}(t,x)/\mathbb{Q}(t)$ is unramified at the preimages $t=0$ and $t=\infty$ of $s=0$ and $s=\infty$, and instead ramified (with inertia group generated by a $p$-cycle) only at the two preimages $t=\pm \sqrt{\frac{1}{cd_i}}$ of the third branch point $s=1$ of $g(X)=s$. Hence, by the Riemann-Hurwitz formula, 
 $\mathbb{Q}(t,x)$ is of genus $0$. Being furthermore an odd degree extension of the rational function field $\mathbb{Q}(t)$ (hence, having a divisor of odd degree), it must itself be a rational function field, i.e., a root field of $f_i(X)-t$ for a suitable rational function $f_i\in \mathbb{Q}(X)$. The ramification type shows furthermore that $f_i$ is linearly related, over $\mathbb{C}$, to $X^p$, compare Proposition \ref{prop:gal_genus0}. Finally, 
 since $\mathrm{Split}(f_i(X)-t)$
 is the compositum of $\mathbb{Q}(t)=\mathbb{Q}(\sqrt{cd_is})$ and the Galois closure of $\mathbb{Q}(x)/\mathbb{Q}(s)$, it contains the biquadratic subextension $\mathbb{Q}(\sqrt{cd_is}, \sqrt{cs})/\mathbb{Q}(s) = \mathbb{Q}(t, \sqrt{d_i})/\mathbb{Q}(s)$, whence $\mathrm{Split}(f_i(X)-t)/\mathbb{Q}(t)$ factors through the constant extension $\mathbb{Q}(\sqrt{d_i})(t)/\mathbb{Q}(t)$. This shows i).

 Furthermore, since each $f_i$ constructed in this way is linearly related over $\mathbb{C}$ to $X^p$, it follows readily that $\mathrm{Gal}(\prod_{i=1}^r (f_i(X)-t)/\overline{\mathbb{Q}}(t)) \le C_p^r$. If the containment were proper, then some $f_i(X)-t$ would have to split in $F_i:=\mathrm{Split}(\prod_{j\ne i}(f_j(X)-t)/\overline{\mathbb{Q}}(t))$. However, since the fields $K_i =\mathbb{Q}(\sqrt{d_i})$ are pairwise linearly disjoint, the sets of branch points $t=\pm \sqrt{\frac{1}{cd_i}}$ of the $f_i$ are disjoint as well, i.e., $F_i/\overline{\mathbb{Q}}(t)$ is unramified at the branch points of $f_i$, and thus $f_i(X)-t$ certainly cannot split in $F_i$. This shows ii) and completes the proof.
 %
 %
 \end{proof}
Functions $f_i$ as constructed in Lemma \ref{lem:twist} are also known as R\'edei functions, cf., e.g., \cite[Section 5]{GMS}.
\begin{proof}[Proof of Lemma \ref{lem:many}]
Choose pairwise linearly disjoint quadratic number fields $K_i=\mathbb{Q}(\sqrt{d_i})$, $i=1,\dots, r$, and set $F_i:=K_i(t)$. 
Associated to each $K_i$ choose a degree-$p$ rational function $f_i\in \mathbb{Q}(X)$ as constructed in Lemma \ref{lem:twist}. Furthermore, let $f_0:=T_p(X)$, and let $G:=\mathrm{Gal}(\prod_{i=0}^{r} (f_i(X)-t)/ \mathbb{Q}(t)) \le AGL_1(p)^{r+1}$. 
By construction, $G$ contains $C_p^{r+1}$. 
Consider the (unique) quadratic subextension $F_0/\mathbb{Q}(t)$ of the splitting field of $f_0(X)-t$. As in the proof of Lemma \ref{lem:cheb}, $F_0$ is a rational function field, of the form $F_0=\mathbb{Q}(\sqrt{\mu(t)})$ for a degree one rational function $\mu\in \mathbb{Q}(t)$. In total, the splitting field $\Omega:=\mathrm{Split}(\prod_{i=0}^{r} (f_i(X)-t)/ \mathbb{Q}(t))$ contains a quadratic subfield $F_{r+1}:=\mathbb{Q}(\sqrt{d_1\cdots d_r \mu(t)})$, which is thus the splitting field of $f_{r+1}(X)-t$ for a suitable quadratic rational function $f_{r+1}\in \mathbb{Q}(X)$. We now claim that every element of $G$ fixes a root of some $f_i(X)-t$, $i\in \{0,\dots,r+1\}$. Indeed, the only elements of $G$ fixing no root of $f_i(X)-t$ for any $0\le i\le r$ are the ones acting as a disjoint product of $p$-cycles in $AGL_1(p)^{r+1}$, but since these are necessarily of odd order, they must fix the roots of $f_{r+1}(X)-t$. Hence, $\{f_0,\dots, f_{r+1}\}$  locally represents $\mathbb{Q}$ by Lemma \ref{lem:transl}. We next claim that no proper subset of $\{f_0,\dots, f_{r+1}\}$ does so. Indeed, $G$ viewed in its action on the roots of $\prod_{i=0}^{r} (f_i(X)-t)$, contains $C_p^{r+1}$, and hence contains a fixed point free element, i.e., $\{f_0,\dots, f_{r}\}$ does not locally represent $\mathbb{Q}$. To see that $\{f_0,\dots, f_{r+1}\}\setminus\{f_i\}$ does not locally represent $\mathbb{Q}$ either for any $0\le i\le r$, fix such an index $i$ and choose $\sigma\in G$ fixing $K_i$ pointwise, but not fixing $F_j$ for any $j\in \{0,\dots, r\}\setminus\{i\}$. Such $\sigma$ exists due to the fields $F_0,\dots, F_r$ being pairwise linearly disjoint quadratic extensions of $\mathbb{Q}(t)$ (indeed, for $F_1,\dots, F_r$, this is by assumption, and $F_0$ is moreover the only nonconstant extension of $\mathbb{Q}(t)$ among the $F_i$, $i\in \{0,\dots, r\})$. But then, by construction, $\sigma$ moves $\sqrt{d_1\cdots d_r \mu(t)}\in F_{r+1}$, i.e., does not fix the roots of $f_{r+1}(X)-t$. This shows the claim and concludes the proof.
\end{proof}

\subsection{Functions with Galois closure of genus $1$}
\label{sec:genus1}
Here we show that (geometrically indecomposable) rational functions $f(X)\in \mathbb{Q}(X)$ such that the splitting field of $f-t$ is of genus $1$ occur as part of a set locally representing $\mathbb{Q}$ as long as their degrees are sufficiently large.
Such functions arise from isogenies or endomorphisms of elliptic curves, as outlined in detail, e.g., in \cite[Section 6]{GMS}.
\begin{theorem}
\label{thm:genus1}
Let $f(X)\in \mathbb{Q}(X)$ be a geometrically indecomposable rational function of degree $>163$ such that the splitting field of $f(X)-t$ is of genus $1$.
Then $f$ is part of a minimal set of (at most four) rational functions locally representing $\mathbb{Q}$.
\end{theorem}
\begin{proof}
From Proposition \ref{prop:gal_genus1} together with Remark \ref{rem:gal_genus1}, such a function $f\in \mathbb{Q}(X)$ is necessarily of degree $p^2$ for a prime $p$, with exactly three branch points, of ramification indices $(3,3,3)$, $(2,4,4)$ or $(2,3,6)$. Furthermore, again by Remark \ref{rem:gal_genus1}, 
the geometric monodromy group $G:=\mathrm{Gal}(f(X)-t/\overline{\mathbb{Q}}(t))$ is isomorphic to $(C_p\times C_p)\rtimes C_d\le AGL_1(\mathbb{F}_{p^2})$ with $d\in \{3,4,6\}$. Moreover, $p\equiv -1$ mod $d$, and the
arithmetic monodromy group $A:=\mathrm{Gal}(f(X)-t/\mathbb{Q}(t))$ is isomorphic to $(C_p\times C_p)\rtimes (C_{p^2-1}\rtimes C_2) = A\Gamma L_1(\mathbb{F}_{p^2})$.

 Since $A/G\cong C_{(p^2-1)/d} \rtimes C_2$ and $(p^2-1)/d$ is necessarily even, $A/G$ has a quotient $C_2\times C_2$, i.e., the splitting field of $f(X)-t$ contains a biquadratic extension $\mathbb{Q}(\sqrt{\alpha}, \sqrt{\beta})/\mathbb{Q}$. Now let us look at the fixed points of elements in $A$. Since $AGL_1(\mathbb{F}_{p^2}) \cong (C_p\times C_p)\rtimes C_{p^2-1}\le S_{p^2}$ is a Frobenius group\footnote{I.e., no nonidentity element fixes more than one point, while elements with one fixed point do exist.} with Frobenius kernel $\mathbb{F}_{p^2}\cong C_p\times C_p$, all its elements outside of $C_p\times C_p$ have exactly one fixed point. In particular, the only fixed point free elements of $A$ are the non-identity elements of $C_p\times C_p$ (in particular fixing $\mathbb{Q}(\sqrt{\alpha}, \sqrt{\beta})$ pointwise), and (certain) elements outside of the index-$2$ normal subgroup $N:=AGL_1(\mathbb{F}_{p^2})$.  Say that (the constant extension) $\mathbb{Q}(\sqrt{\alpha})(t)$ is the fixed field of $N$, so that the latter elements necessarily move $\sqrt{\alpha}$.  

Now choose any prime $q\equiv 3$ mod $4$. 
Using Lemma \ref{lem:twist}, choose rational functions $f_0$ and $f_1\in \mathbb{Q}(X)$, linearly related to $X^q$ over $\mathbb{C}$ and with $\mathrm{Split}(f_i(X)-t/\mathbb{Q}(t))$ containing $\sqrt{\alpha\beta}$ resp.\ $\sqrt{\alpha}$ for $i=0$ resp.\ $i=1$. 
Then as in the proof of Lemma \ref{lem:many}, one obtains a minimal set $\{f_0,f_1,f_2, T_q\}$ of rational functions locally representing $\mathbb{Q}$, where moreover $f_2$ is quadratic with $\mathrm{Split}(f_2(X)-t/\mathbb{Q}(t))=\mathbb{Q}(\sqrt{\beta\mu(t)})$ with $\mu(t)\in \mathbb{Q}(t)$ the degree-$1$ rational function such that $\mathrm{Split}(T_q(X)-t/\mathbb{Q}(t))\supset \mathbb{Q}(\sqrt{\mu(t)})$.
We claim that we can replace $f_0$ by $f$ without losing the locally representing property. This then finishes the proof, since $\{f_1,f_2, T_q\}$ is not locally representing by construction. To show the claim, we need to show that all automorphisms $\sigma$ of $\mathrm{Split}((f(X)-t)(f_1(X)-t)(f_2(X)-t)(T_q(X)-t))/\mathbb{Q}(t)$ not fixing a root of $f(X)-t$ must necessarily fix a root of $f_i(X)-t$ for some $i\in \{1,2\}$ or of $T_q(X)-t$. But as noted above, there are only two possibilities for such $\sigma$: either $\sigma(\sqrt{\alpha}) =  -\sqrt{\alpha}$, in which case $\sigma$ necessarily restricts to an element of of even order  (i.e., having a fixed point) in $\mathrm{Gal}(f_1(X)-t/\mathbb{Q}(t)) \cong AGL_1(q)$. Or $\sigma(\sqrt{\beta})=\sqrt{\beta}$, in which case there are two further possibilities: if $\sigma$ also fixes $\sqrt{\mu(t)}$, then it fixes the roots of $f_2$; if, conversely, it does not fix $\sqrt{\mu(t)}$, then it acts as an element of even order on the roots of $T_q$, and hence has a fixed point.
This concludes the proof.
\end{proof}

\subsection{Proof of Theorem \ref{thm:converse}}
The examples collected in the previous sections are already enough to derive Theorem \ref{thm:converse}. We summarize the reasoning behind this. Firstly, a geometrically indecomposable rational function $f(X)\in \mathbb{C}(X)$ whose splitting field is still of genus $0$ is known to be linearly related (over $\mathbb{C}$) to $X^p$, $T_p(X)$ (for some prime $p$) or to an indecomposable subcover of the tetrahedral, octahedral or icosahedral Galois covers $\mathbb{P}^1\to \mathbb{P}^1$, see Proposition \ref{prop:gal_genus0}. The cases $X^p$ and $T_p(X)$ are covered by Lemma \ref{lem:cheb}. 
In fact, the icosahedral cases are covered by Example \ref{ex:sporadic} and the octahedral ones by Lemma \ref{lem:symmetric}a); the interested reader may want to verify that the tetrahedral case can be handled similarly, so that indeed the cases with splitting field of genus $0$ occur (up to $\mathbb{C}$-linear equivalence) without exception. Combine this with Theorem \ref{thm:genus1} to readily obtain Theorem \ref{thm:converse}.

\begin{question}
Can one always choose $\tilde{f}=f$ in \ref{thm:converse}?
 In several individual cases, this is actually true. However, notably the following case might be problematic: Let $p\equiv 3$ mod $4$ be a prime and $f(X)\in \mathbb{Q}(X)$ a degree-$p$ rational function totally ramified at exactly two algebraically conjugate points, with monodromy group $\Mon(f) = C_p\rtimes C_{(p-1)/2}$. Such a function (called a R\'edei function) exists for all such $p$ and can be constructed via pullback from the Chebyshev polynomial $T_p$ similarly as the functions $f_i$ in Lemma \ref{lem:many}. Evidently, such $f$ is linearly related to $X^p$ over $\mathbb{C}$, but not over $\mathbb{Q}$. Via choice of $p$ in suitable residue classes, we may assume the smallest prime divisor of $\frac{p-1}{2}$ to be arbitrarily large, and it is even expected that this value itself is a (Germain) prime infinitely often. Our approaches above, which rely crucially on the existence of certain small degree subfields in the Galois closure, do not carry over to this case; is it still possible to find an alternative construction?
\end{question}

\textbf{Acknowledgement:} The first-named author acknowledges the support of the Austrian Science Fund (FWF): W1230. The second-named author was supported by the National Research Foundation of Korea (NRF Basic Research Grant RS-2023-00239917). We thank Danny Neftin and Shai Rosenberg for some helpful comments, as well as the two referees for many valuable suggestions for improvement.

\end{document}